\documentclass{amsart}
\usepackage{pgf, tikz}
\usetikzlibrary{calc}
\usepackage{amsmath,color}

\newtheorem{theorem}{Theorem}[section]

\newtheorem{corollary}[theorem]{Corollary}

\theoremstyle{definition}

\theoremstyle{remark}
\newtheorem{remark}[theorem]{Remark}

\numberwithin{equation}{section}

\definecolor{red}{rgb}{1,0,0}
\definecolor{blue}{rgb}{.2,.2,.8}

\begin{document}

\title{A generalization of complete and elementary symmetric functions}

\author[M.~Ahmia]{Moussa Ahmia}
\address[M.~Ahmia]{Department of Mathematics, Mohamed Seddik Benyahia University, Jijel, Algeria}
\curraddr{LMAM Laboratory, BP 18000 Ouled Aissa, Jijel}
\email{ahmiamoussa@gmail.com or moussa.ahmia@univ-jijel.dz}
\author[M.~Merca]{Mircea Merca}
\address[M.~Merca]{Department of Mathematics, University of Craiova, Craiova, Romania\\Academy of Romanian Scientists, Ilfov 3, Sector 5, Bucharest, Romania}
\email{mircea.merca@profinfo.edu.ro}


\date{4 May 2020}



\begin{abstract}
In this paper, we consider the generating functions of the complete and elementary symmetric functions and provide a new generalization of these classical symmetric functions. Some classical relationships involving the complete and elementary symmetric functions are reformulated in a more general context.
Combinatorial interpretations of these generalized symmetric functions are also introduced.
\end{abstract}

\maketitle

\section{Introduction}
\label{S1}

A formal power series in the variables $x_1,x_2,\ldots,x_n$ is called symmetric if it is
invariant under any permutation of the variables. These symmetric formal power
series are traditionally called symmetric functions. A symmetric function is homogeneous of degree $k$ if every monomial in it has total degree $k$.
Symmetric functions are ubiquitous in mathematics and mathematical physics. For example, they appear in elementary algebra (e.g. Viete’s Theorem), representation theories of symmetric groups and general linear groups over $\mathbb{C}$ or finite fields. They are also important objects to study in algebraic combinatorics.

A partition $\lambda=[\lambda_1,\lambda_2,\ldots,\lambda_k]$ of a positive integer $n$ is a weakly decreasing sequence of positive integers whose sum is $n$, i.e.,
$$\lambda_1+\lambda_2+\cdots+\lambda_k = n\qquad\text{and}\qquad \lambda_1\ge\lambda_2\ge\cdots\ge\lambda_k>0.$$
The positive integers in the sequence are called parts \cite{BA0}.
The multiplicity of the part $i$ in $\lambda$, denoted by $t_i$, is the number of parts of $\lambda$ equal to $i$.
We denote by $l(\lambda)$ the number of parts of $\lambda$.
In order to indicate that $\lambda=[\lambda_1,\lambda_2,\ldots,\lambda_k]$
or $\lambda=[1^{t_1} 2^{t_2} \ldots n^{t_n}]$
is a partition of $n$, we use the notation $\lambda \vdash n$.
Very recently, Merca \cite{Mer1,Mer2} gave the fastest algorithms to enumerate all the partitions of an integer.

We recall some basic facts about monomial symmetric functions. Proofs and details
can be found in Macdonald’s book  \cite{Mck}. If $\lambda =[\lambda _1,\lambda _2 ,\ldots,\lambda_k]$ is an integer partition with $k\leq n$ then,  the monomial symmetric function
$$m_{\lambda}(x_1,x_2,\ldots,x_n)=m_{[\lambda_1,\lambda_2,\ldots,\lambda_k]}(x_1,x_2,\ldots,x_n)$$
is the sum of monomial $x_1^{\lambda_1}x_2^{\lambda_2}\cdots x_k^{\lambda_k}$ and all distinct monomials obtained from it by a permutation of variables. For instance, with $\lambda=[2,1,1]$ and $n=4$, we have:
\begin{align*}
 m_{[2,1,1]}(x_1,x_2,x_3,x_4) & =x_1^2x_2x_3+x_1x_2^2x_3+x_1x_2x_3^2+x_1^2x_2x_4\\
& \qquad +x_1x_2^2x_4+x_1x_2x_4^2+x_1^2x_3x_4+x_1x_3^2x_4\\
& \qquad\qquad +x_1x_3x_4^2+x_2^2x_3x_4+x_2x_3^2x_4+x_2x_3x_4^2.
\end{align*}

The $k$th complete homogeneous symmetric function $h_k$ is the sum of all
monomials of total degree $k$ in these variables, i.e.,
\begin{equation*}
h_{k}(x_{1},x_{2},\ldots,x_{n})=\sum_{\lambda \vdash k} m_{\lambda}(x_{1},x_{2},\ldots,x_{n})=\sum_{1\leq i_{1}\leq
	i_{2}\leq \cdots \leq i_{k}\leq n}x_{i_{1}}x_{i_{2}}\cdots x_{i_{k}},
\end{equation*}
and the $k$th elementary symmetric function is defined by
\begin{equation*}
e_{k}(x_{1},x_{2},\ldots,x_{n})=m_{[1^{k}]}(x_{1},x_{2},\ldots,x_{n})=\sum_{1\leq
	i_{1}<i_{2}<\cdots <i_{k}\leq n}x_{i_{1}}x_{i_{2}}\cdots x_{i_{k}},
\end{equation*}
where $e_{0}(x_{1},x_{2},\ldots,x_{n})=h_{0}(x_{1},x_{2},\ldots,x_{n})=1$.
In particular, when $\lambda=[k]$, we have the $k$th power sum symmetric function
$$
p_k(x_{1},x_{2},\ldots,x_{n}) =m_{[k]}(x_{1},x_{2},\ldots,x_{n})=\sum_{i=1}^{n}x_{i}^k,
$$
with $p_0(x_{1},x_{2},\ldots,x_{n})=n$.

The complete homogeneous symmetric functions are characterized by the following identity of
formal power series in $t$:
\begin{equation}\label{e1}
\sum_{k=0}^\infty h_{k}(x_{1},x_{2},\ldots,x_{n})t^{k}=\prod_{i=1}^{n}(1-x_{i}t)^{-1}.
\end{equation}
Analogous, for the elementary symmetric functions we have:
\begin{equation}\label{e2}
\sum_{k=0}^\infty e_{k}(x_{1},x_{2},\ldots,x_{n})t^{k}=\prod_{i=1}^{n}(1+x_{i}t).
\end{equation}

Inspired by these generating functions, we introduce the generalized
symmetric functions $H_{k}^{(s)}(x_{1},x_{2},\ldots,x_{n})$ and  $E_{k}^{(s)}(x_{1},x_{2},\ldots,x_{n})$ as follows:
\begin{equation}\label{Eq1}
\sum_{k=0}^\infty H_{k}^{(s)}(x_{1},x_{2},\ldots,x_{n})t^{k}
 =\prod_{i=1}^{n}\big(1-x_{i}t+\cdots +(-x_it)^s\big)^{-1}
\end{equation}
and
\begin{equation}\label{Eq2}
\sum_{k=0}^\infty E_{k}^{(s)}(x_{1},x_{2},\ldots,x_{n})t^{k}
=\prod_{i=1}^{n}\big(1+x_{i}t+\cdots +(x_it)^s\big),
\end{equation}
where $s$ is a positive integer.

Clearly, by setting $s=1$ in \eqref{Eq1} and \eqref{Eq2}, we obtain the generating functions for complete and elementary symmetric functions.  In addition, by  \eqref{Eq2} we easily deduce that
\begin{equation}\label{Eq3}
E_{k}^{(s)}(x_{1},x_{2},\ldots ,x_{n})=\sum_{\substack{ %
		\lambda \vdash k \\ \lambda_1\leq s}} m_{\lambda}(x_1,x_2,\ldots,x_n).
\end{equation}
Moreover, considering that
\begin{equation}\label{Eq4}
E_{k}^{(k)}(x_{1},x_{2},\ldots,x_{n}) = h_{k}(x_{1},x_{2},\ldots,x_{n}),
\end{equation}
the generalized symmetric functions $E_{k}^{(s)}(x_{1},x_{2},\ldots ,x_{n})$ can be seen as another generalization of the complete homogenous symmetric function $h_{k}(x_{1},x_{2},\ldots,x_{n})$.
To illustrate \eqref{Eq3}, we have
\begin{align*}
E_{5}^{(3)}(x_{1},x_{2},x_{3})&=m_{[2,2,1]}(x_{1},x_{2},x_{3})+m_{[3,1,1]}(x_{1},x_{2},x_{3})+m_{[3,2]}(x_{1},x_{2},x_{3}) \\
&=x_{1}^{2}x_{2}^{2}x_{3}+x_{1}^{2}x_{2}x_{3}^{2}+x_{1}x_{2}^{2}x_{3}^{2}+x_{1}^{3}x_{2}x_{3}
+x_{1}x_{2}^{3}x_{3}+x_{1}x_{2}x_{3}^{3}\\
&\qquad +x_{1}^{2}x_{2}^{3}+x_{1}^{2}x_{3}^{3}+x_{2}^{2}x_{3}^{3}+x_{1}^{3}x_{2}^{2}+x_{1}^{3}x_{3}^{2}+x_{2}^{3}x_{3}^{2}.
\end{align*}

We remark that the symmetric functions $E_{k}^{(s)}(x_{1},x_{2},\ldots ,x_{n})$
are not essentially a new generalization of the elementary symmetric functions $e_{k}(x_{1},x_{2},\ldots ,x_{n})$. An equivalent definition of these symmetric functions already exists in a paper published in $2018$ by Bazeniar et \textit{al.} \cite{BAZ}:
\begin{equation}\label{Eq6}
E_{k}^{(s)}(x_{1},x_{2},\ldots ,x_{n})=\sum_{\substack{
		\lambda \vdash k \\ 0\leq \lambda
		_{1},\lambda _{2},\ldots ,\lambda _{n}\leq s}}x_{1}^{\lambda
	_{1}}x_{2}^{\lambda _{2}}\cdots x_{n}^{\lambda_{n}},
\end{equation}
where $E_{0}^{(s)}(x_1,x_2,\ldots,x_n)=1$ and $E_{k}^{(s)}(x_1,x_2,\ldots,x_n)=0$ unless $0\leq k\leq sn$.
Moreover, the authors proved that the symmetric functions
$E_{k}^{(s)}(x_{1},x_{2},\ldots ,x_{n})$  satisfy the following recurrence relation.
\begin{equation}\label{Eq7}
	E_{k}^{(s)}(x_{1},x_{2},\ldots,x_{n})
	=\sum_{j=0}^{s}x_{n}^{j}E_{k-j}^{(s)}(x_{1},x_{2},\ldots,x_{n-1}).
\end{equation}
A similar result can be easily derived for the symmetric function
$H_{k}^{(s)}(x_{1},x_{2},\ldots ,x_{n})$, namely,
\begin{equation}\label{Eq8}
H_{k}^{(s)}(x_{1},x_{2},\ldots,x_{n-1})=\sum_{j=0}^{s} (-1)^j x_{n}^{j} H_{k-j}^{(s)}(x_{1},x_{2},\ldots,x_{n}).
\end{equation}

Very recently,  Fu and Mei \cite{Fu} and Grinberg \cite{Gri} independently introduced
the generalized symmetric functions $E_k^{(s)}$. Grinberg denoted these functions by $G(s,k)$ and called the \textit{Petrie symmetric functions} while
Fu and Mei used the notation $h_k^{[s]}$ and referred to them as \textit{truncated homogeneous symmetric functions}.
It seems that the paper \cite{BAZ} is not known by the authors of \cite{Fu,Gri}.

In this paper, motivated by these results, we investigate the properties of the generalized symmetric functions $H_{k}^{(s)}(x_{1},x_{2},\ldots,x_{n})$ and  $E_{k}^{(s)}(x_{1},x_{2},\ldots,x_{n})$.
Our paper is structured as follows. In the next section, we collect some classical relationships involving complete, elementary and power sum symmetric functions and provide generalizations for them. In Section \ref{S3} we show that the generalized symmetric functions $H_{k}^{(s)}$ and $E_{k}^{(s)}$
can be expressed in terms of the complete and elementary symmetric functions. In Section \ref{S4} we consider some combinatorial interpretations for the generalized symmetric functions  $H_{k}^{(s)}$ and $E_{k}^{(s)}$.

\section{Newton identities revisited}
\label{S2}

There is a fundamental relation between the elementary symmetric polynomials and the complete homogeneous ones:
\begin{equation}\label{E1}
\sum_{j=0}^{k} (-1)^j e_j(x_1,x_2,\ldots,x_n)h_{k-j}(x_1,x_2,\ldots,x_n)=\delta_{0,k},
\end{equation}
where $\delta_{i,j}$ is the Kronecker delta.
This relation is valid for all $k>0$, and any number of variables $n$.
We have the following generalization of this classical identity.

\begin{theorem}\label{T2.1}
	Let $k$, $n$ and $s$ be three positive integers and let $x_1,x_2,\ldots,x_n$ be independent
	variables. Then
	\begin{equation}
	\sum_{j=0}^{k} (-1)^j E_j^{(s)}(x_1,x_2,\ldots,x_n)H_{k-j}^{(s)}(x_1,x_2,\ldots,x_n)=\delta_{0,n}.
	\end{equation}
\end{theorem}

\begin{proof}
	By \eqref{Eq1} and \eqref{Eq2}, we see that
	$$
	\left( \sum_{k=0}^\infty (-1)^k E_{k}^{(s)}(x_{1},x_{2},\ldots,x_{n})t^{k} \right) \left( \sum_{k=0}^\infty H_{k}^{(s)}(x_{1},x_{2},\ldots,x_{n})t^{k} \right) = 1.
	$$
	Considering the well known Cauchy product of two power series, we obtain
	$$
	\sum_{k=0}^\infty \left( \sum_{j=0}^{k} (-1)^j E_j^{(s)}(x_1,x_2,\ldots,x_n)H_{k-j}^{(s)}(x_1,x_2,\ldots,x_n) \right) t^k = 1.
	$$
	This concludes the proof.
\end{proof}

Theorem \ref{T2.1} and \cite[Theorem 1]{Mer14} allow us to derive two symmetric identities for the generalized symmetric functions $H_{k}^{(s)}$ and $E_{k}^{(s)}$.

\begin{corollary}
	Let $k$, $n$ and $s$ be three positive integers and let $x_1,x_2,\ldots,x_n$ be independent
	variables. The generalized symmetric functions $H_{k}^{(s)}=H_{k}^{(s)}(x_1,x_2,\ldots,x_n)$ and $E_{k}^{(s)}=E_{k}^{(s)}(x_1,x_2,\ldots,x_n)$ are related by
	\[
	H_{k}^{(s)}=\sum_{t_{1}+2t_{2}+\cdots +kt_{k}=k}(-1)^{k+t_{1}+t_{2}+\cdots +t_{k}} \binom{t_{1}+t_{2}+\cdots +t_{k}}{t_{1},t_{2},\ldots ,t_{k}}
	\prod_{i=1}^k \left( E_{i}^{(s)}\right) ^{t_{i}}
	\]%
	and
	\[
	E_{k}^{(s)}=\sum_{t_{1}+2t_{2}+\cdots +kt_{k}=k}(-1)^{k+t_{1}+t_{2}+\cdots +t_{k}} \binom{t_{1}+t_{2}+\cdots +t_{k}}{t_{1},t_{2},\ldots ,t_{k}}
	\prod_{i=1}^k\left( H_{i}^{(s)}\right) ^{t_{i}}.
	\]
\end{corollary}

The problem of expressing power sum symmetric polynomials in terms of elementary symmetric polynomials and vice-versa and the problem of expressing power sum symmetric polynomials in terms of complete symmetric polynomials and vice-versa were solved a long time ago.
The relations as Newton's identities
\begin{equation}\label{E2}
ke_k(x_1,x_2,\ldots,x_n)=\sum_{j=1}^{k}(-1)^{j-1}e_{k-j}(x_1,x_2,\ldots,x_n)p_j(x_1,x_2,\ldots,x_n)
\end{equation}
or
\begin{equation}\label{E3}
kh_k(x_1,x_2,\ldots,x_n)=\sum_{j=1}^{k}h_{k-j}(x_1,x_2,\ldots,x_n)p_j(x_1,x_2,\ldots,x_n)
\end{equation}
are well known. Recently, Merca \cite{Mer16a} proved that the complete, elementary and power sum symmetric functions are related by
\begin{equation}\label{E4}
p_k(x_1,x_2,\ldots,x_n)=\sum_{j=1}^{k}(-1)^{j-1}je_j(x_1,x_2,\ldots,x_n)h_{k-j}(x_1,x_2,\ldots,x_n)
\end{equation}
and derived new relationships between complete and elementary symmetric functions:
\begin{equation}\label{E5}
2ke_k = \sum_{k_1+k_2+k_3=k} (-1)^{k_3} (k_1+k_2)e_{k_1}e_{k_2}h_{k_3},
\end{equation}
and
\begin{equation}\label{E6}
kh_k = \sum_{k_1+k_2+k_3=k} (-1)^{k_3-1} k_3 h_{k_1}h_{k_2}e_{k_3},
\end{equation}
where $k_1,k_2,k_3$ are nonnegative integers.

In order to provide the generalizations of \eqref{E2}-\eqref{E6}, we consider the symmetric function $P_k^{(s)}$ defined as
\begin{equation*}
P_{k}^{(s)}(x_{1},x_{2},\ldots,x_{n}) = c_k^{(s)} p_{k}(x_{1},x_{2},\ldots,x_{n}),
\end{equation*}
where
$$c_k^{(s)}=\begin{cases}
(-1)^{k} \cdot s,& \text{if $k\equiv 0 \bmod {(s+1)}$,}\\
(-1)^{k-1}, & \text{otherwise.}
\end{cases}
$$

\begin{theorem}\label{T2.3}
	Let $k$, $n$ and $s$ be three positive integers and let $x_1,x_2,\ldots,x_n$ be independent
	variables. Then
	\begin{enumerate}
		\item $\displaystyle{ kE_{k}^{(s)}(x_{1},x_{2},\ldots,x_{n})
			=\sum_{j=1}^{k}(-1)^{j-1} P_{j}^{(s)}(x_{1},x_{2},\ldots,x_{n}) E_{k-j}^{(s)}(x_{1},x_{2},\ldots,x_{n}) };$
		\item $\displaystyle{ kH_{k}^{(s)}(x_{1},x_{2},\ldots,x_{n})
			=\sum_{j=1}^{k} P_{j}^{(s)}(x_{1},x_{2},\ldots,x_{n}) H_{k-j}^{(s)}(x_{1},x_{2},\ldots,x_{n}) };$
		\item $\displaystyle{ P_{k}^{(s)}(x_{1},x_{2},\ldots,x_{n})
			=\sum_{j=1}^{k} (-1)^{j-1}j E_{j}^{(s)}(x_{1},x_{2},\ldots,x_{n}) H_{k-j}^{(s)}(x_{1},x_{2},\ldots,x_{n}) }.$
	\end{enumerate}
\end{theorem}

\begin{proof}
	\allowdisplaybreaks{
	For $\omega_{j,s}=e^{2j \pi i/s}$ with $j=1,2,\ldots, s-1$, we can see that
	\begin{align*}
	& 1-t+\cdots+(-t)^{s-1}  = (-1)^{s-1} \prod_{j=1}^{s-1} (\omega_{j,s}+t)\\
	& \qquad= (-1)^{s-1} \prod_{j=1}^{s-1} \omega_{j,s}\left( 1+\frac{t}{\omega_{j,s}}\right)  = \prod_{j=1}^{s-1} \left( 1+\omega_{j,s} t\right),
	\end{align*}
	where we take into account that
	$$\prod_{j=1}^{s-1} \omega_{j,s} = (-1)^{s-1}\qquad\text{and}\qquad \omega_{j,s} = \frac{1}{\omega_{s-j,s}}.$$
	On one hand, we have
	\begin{align}
	& \frac{d}{dt} \ln \prod_{i=1}^{n}\big(1-x_{i}t+\cdots +(-x_it)^s\big)^{-1}\label{Eq2.8} \\
	&\qquad = \frac{d}{dt} \ln \prod_{i=1}^{n} \prod_{j=1}^s \left(1+\omega_{j,s+1}x_i t \right)^{-1}\notag \\
	& \qquad = \sum_{i=1}^n \sum_{j=1}^s \frac{d}{dt} \ln \left(1+\omega_{j,s+1}x_i t \right)^{-1}\notag \\
	& \qquad = -\sum_{i=1}^n \sum_{j=1}^s \frac{\omega_{j,s+1} x_i}{1+\omega_{j,s+1}x_i t }\notag \\
	& \qquad = -\sum_{j=1}^s \sum_{i=1}^n \left( \omega_{j,s+1} x_i - (\omega_{j,s+1} x_i)^2 t + (\omega_{j,s+1} x_i)^3 t^2 - \cdots \right)\notag  \\
	& \qquad = \sum_{k=1}^\infty (-1)^k \Bigg(\sum_{j=1}^s \omega_{j,s+1}^k \Bigg)
	\Bigg( \sum_{i=1}^n x_i^k \Bigg) t^{k-1}\notag \\
	& \qquad = \sum_{k=1}^\infty (-1)^k p_k\left(\omega_{1,s+1},\omega_{2,s+1},\ldots,\omega_{s,s+1}  \right)
	p_k(x_1,x_2,\ldots,x_n) t^{k-1}\notag \\
	& \qquad = \sum_{k=1}^\infty P_k^{(s)} (x_1,x_2,\ldots,x_n) t^{k-1},\notag
	\end{align}
	where we have invoked that
	$$ p_k(\omega_{1,s+1},\omega_{2,s+1},\ldots,\omega_{s,s+1})
	   = \begin{cases}
	   s,& \text{if $k\equiv 0 \bmod {(s+1)}$,}\\
	   -1, & \text{otherwise.}
	   \end{cases}
	$$
	
	On the other hand, we can write
		\begin{align}
		& \sum_{k=1}^\infty P_k^{(s)} (x_1,x_2,\ldots,x_n) t^{k-1}\notag \\
		& \qquad = \frac{d}{dt} \ln \prod_{i=1}^{n}\big(1-x_{i}t+\cdots +(-x_it)^s\big)^{-1}\notag \\
		& \qquad = \frac{d}{dt} \ln \left( \sum_{k=0}^{\infty} (-1)^k E_k^{(s)} (x_1,x_2,\ldots,x_n) t^{k} \right)^{-1}\notag \\
		& \qquad = -\dfrac{\sum\limits_{k=1}^{\infty} (-1)^{k} k E_k^{(s)} (x_1,x_2,\ldots,x_n) t^{k-1}}{\sum\limits_{k=0}^{\infty} (-1)^k E_k^{(s)} (x_1,x_2,\ldots,x_n) t^{k}}. \label{E7}
		\end{align}
		By this identity, with $t$ replaced by $-t$, we obtain
		\begin{align*}
		& \sum_{k=1}^{\infty} k E_k^{(s)} (x_1,x_2,\ldots,x_n) t^{k-1}  \\
		& \qquad = \left( \sum_{k=0}^{\infty} E_k^{(s)} (x_1,x_2,\ldots,x_n) t^{k} \right)
		 \left( \sum_{k=1}^\infty (-1)^{k-1} P_k^{(s)} (x_1,x_2,\ldots,x_n) t^{k-1} \right)
		\end{align*}
	and the first identity follows.

	To prove the second identity, we consider
	\begin{align*}
	& \sum_{k=1}^\infty P_k^{(s)} (x_1,x_2,\ldots,x_n) t^{k-1}\\
	& \qquad = \frac{d}{dt} \ln \prod_{i=1}^{n}\big(1-x_{i}t+\cdots +(-x_it)^s\big)^{-1}\\
	& \qquad = \frac{d}{dt} \ln \sum_{k=0}^{\infty} H_k^{(s)} (x_1,x_2,\ldots,x_n) t^{k}\\
	& \qquad = \left( \sum_{k=1}^{\infty} k H_k^{(s)} (x_1,x_2,\ldots,x_n) t^{k-1}\right) \left( \sum_{k=0}^{\infty} H_k^{(s)} (x_1,x_2,\ldots,x_n) t^{k} \right)^{-1}.
	\end{align*}
	
	Rewriting \eqref{E7} as
	\begin{align*}
	& \sum_{k=1}^\infty P_k^{(s)} (x_1,x_2,\ldots,x_n) t^{k-1}\\
	& \qquad = \left( \sum_{k=1}^{\infty} (-1)^{k-1} k E_k^{(s)} (x_1,x_2,\ldots,x_n) t^{k-1}\right) \left( \sum_{k=0}^{\infty} H_k^{(s)} (x_1,x_2,\ldots,x_n) t^{k} \right),
	\end{align*}	
	we derive the last identity and the theorem is proved.
}
\end{proof}

\begin{corollary}
	Let $k$, $n$ and $s$ be three positive integers and let $x_1,x_2,\ldots,x_n$ be independent
	variables.
	The symmetric functions $E_k^{(s)}=E_k^{(s)}(x_1,x_2,\ldots,x_n)$ and $H_k^{(s)}=H_k^{(s)}(x_1,x_2,\ldots,x_n)$ are related by
	$$
	2k E_k^{(s)} = \sum_{k_1+k_2+k_3=k} (-1)^{k_3} (k_1+k_2) E_{k_1}^{(s)} E_{k_2}^{(s)} H_{k_3}^{(s)}
	$$
	and
	$$
	k H_k^{(s)} = \sum_{k_1+k_2+k_3=k} (-1)^{k_3-1} k_3 H_{k_1}^{(s)} H_{k_2}^{(s)} E_{k_3}^{(s)},
	$$
	where $k_1,k_2,k_3$ are nonnegative integers.
\end{corollary}

It is well-known that the power sum symmetric functions $p_k=p_k(x_1,x_2,\ldots,x_n)$ can be expressed in terms of elementary symmetric functions $e_k=e_k(x_1,x_2,\ldots,x_n)$ using Girard-Newton-Waring formula \cite[eq. 8]{Gou}, namely
\begin{equation*}
p_k = \sum_{t_1+2t_2+\ldots+kt_k=k} \frac{(-1)^{k+t_1+t_2+\cdots+t_k}\cdot k}{t_1+t_2+\cdots+t_k} \binom{t_1+t_2+\cdots+t_k}{t_1,t_2,\ldots,t_k} e_1^{t_1}e_2^{t_2}\cdots e_k^{t_k}.
\end{equation*}
There is a very similar result which combines the power sum symmetric functions $p_k=p_k(x_1,x_2,\ldots,x_n)$ and the complete homogeneous symmetric functions  $h_k=h_k(x_1,x_2,\ldots,x_n)$, i.e.,
\begin{equation*}
p_k = \sum_{t_1+2t_2+\ldots+kt_k=k} \frac{(-1)^{1+t_1+t_2+\cdots+t_k}\cdot k}{t_1+t_2+\cdots+t_k} \binom{t_1+t_2+\cdots+t_k}{t_1,t_2,\ldots,t_k} h_1^{t_1}h_2^{t_2}\cdots h_k^{t_k}.
\end{equation*}
The following two theorems provide generalizations of these relations.

\begin{theorem}\label{T2.5}
	Let $k$, $n$ and $s$ be three positive integers and let $x_1,x_2,\ldots,x_n$ be independent
	variables. The power sum symmetric function $p_k=p_k(x_1,x_2,\ldots,x_n)$ and the generalized symmetric functions $E_k^{(s)}=E_k^{(s)}(x_1,x_2,\ldots,x_n)$ and $H_k^{(s)}=H_k^{(s)}(x_1,x_2,\ldots,x_n)$ are related by
	\begin{align*}
	&  p_k = \frac{\sum\limits_{t_1+2t_2+\cdots+kt_k=k} \dfrac{(-1)^{t_1+t_2+\cdots+t_k}}{t_1+t_2+\cdots+t_k}
		\dbinom{t_1+t_2+\cdots+t_k}{t_1,t_2,\ldots,t_k} \prod\limits_{i=1}^{k} \left(E_i^{(s)} \right)^{t_i} }
	{\sum\limits_{t_1+2t_2+\cdots+st_s=k} \dfrac{(-1)^{t_1+t_2+\cdots+t_s}}{t_1+t_2+\cdots+t_s}
		\dbinom{t_1+t_2+\cdots+t_s}{t_1,t_2,\ldots,t_s}}
	\end{align*}
	and
	\begin{align*}
	&  p_k = \frac{\sum\limits_{t_1+2t_2+\cdots+kt_k=k} \dfrac{(-1)^{1+t_1+t_2+\cdots+t_k}}{t_1+t_2+\cdots+t_k}
		\dbinom{t_1+t_2+\cdots+t_k}{t_1,t_2,\ldots,t_k} \prod\limits_{i=1}^{k} \left(H_i^{(s)} \right)^{t_i} }
	{\sum\limits_{t_1+2t_2+\cdots+st_s=k} \dfrac{(-1)^{k+t_1+t_2+\cdots+t_s}}{t_1+t_2+\cdots+t_s}
		\dbinom{t_1+t_2+\cdots+t_s}{t_1,t_2,\ldots,t_s}}.
	\end{align*}
\end{theorem}

\begin{proof}
	\allowdisplaybreaks{
		Considering \eqref{Eq2} and the logarithmic series
		$$ \ln(1+t) = \sum_{n=1}^\infty (-1)^{n-1} \frac{t^n}{n},\qquad |t|<1,$$
		we can write
		\begin{align*}
		& \ln\left(\sum_{k=0}^\infty E_k^{(s)}(x_1,x_2,\ldots,x_n) t^k \right) \\
		& \quad = \ln \prod_{i=1}^n \big(1+x_it+\cdots +(x_i t)^s\big)\notag \\
		& \quad = \sum_{i=1}^{n} \ln \big(1+x_it+\cdots +(x_it)^s\big)\notag \\
		& \quad = \sum_{i=1}^{n} \sum_{m=1}^\infty \frac{(-1)^{m-1}}{m} \left( \sum_{j=1}^s (x_i t)^j \right)^m\notag \\
		& \quad = \sum_{m=1}^\infty \frac{(-1)^{m-1}}{m} \sum_{i=1}^{n} \sum_{k=m}^{sm} \sum_{\substack{t_1+t_2+\cdots+t_s=m\\t_2+2t_2+\cdots st_s = k}}
		\binom{t_1+t_2+\cdots+t_s}{t_1,t_2,\ldots,t_s} (x_i t)^k\notag \\
		& \quad = \sum_{m=1}^\infty \frac{(-1)^{m-1}}{m} \sum_{k=m}^{sm} \sum_{\substack{t_1+t_2+\cdots+t_s=m\\t_2+2t_2+\cdots st_s = k}}
		\binom{t_1+t_2+\cdots+t_s}{t_1,t_2,\ldots,t_s} p_k(x_1,x_2,\ldots,x_n) t^k\notag \\
		& \quad =\sum_{k=1}^\infty \sum_{t_1+2t_2+\cdots+st_s=k} \frac{(-1)^{1+t_1+t_2+\cdots+t_s}}{t_1+t_2+\cdots+t_s}
		\binom{t_1+t_2+\cdots+t_s}{t_1,t_2,\ldots,t_s} p_k(x_1,x_2,\ldots,x_n) t^k.\notag
		\end{align*}
		On the other hand, we have
		\begin{align}
		& \ln\left(1+\sum_{k=1}^\infty E_k^{(s)}(x_1,x_2,\ldots,x_n) t^k \right)\label{Eq2.10} \\
		&  \quad = \sum_{m=1}^\infty \frac{(-1)^{m-1}}{m} \left( \sum_{k=1}^\infty E_k^{(s)}(x_1,x_2,\ldots,x_n) t^k \right)^m \notag \\
		&  \quad = \sum_{m=1}^\infty \frac{(-1)^{m-1}}{m} \sum_{k=1}^\infty
		\sum_{\substack{t_1+t_2+\cdots+t_k=m\\t_1+2t_2+\cdots+kt_k=k}}
		\binom{t_1+t_2+\cdots+t_k}{t_1,t_2,\ldots,t_k}
		\prod_{i=1}^k \left( E_i^{(s)} \right)^{t_i} t^k\notag \\
		&  \quad = \sum_{k=1}^\infty \sum_{t_1+2t_2+\cdots+kt_k=k}
		\frac{(-1)^{1+t_1+t_2+\cdots+t_k}}{t_1+t_2+\cdots+t_k}
		\binom{t_1+t_2+\cdots+t_k}{t_1,t_2,\ldots,t_k}
		\prod_{i=1}^k \left( E_i^{(s)} \right)^{t_i} t^k.\notag
		\end{align}
		and the first identity follows easily. In a similar way, considering \eqref{Eq1} we can prove the second identity. We obtain
		\begin{align*}
		& \ln\left(\sum_{k=0}^\infty H_k^{(s)}(x_1,x_2,\ldots,x_n) t^k \right) \\
		& \quad = \ln \prod_{i=1}^n \big(1+(-x_it)+\cdots +(-x_i t)^s\big)^{-1} \\
		& \quad = -\sum_{i=1}^{n} \sum_{m=1}^\infty \frac{(-1)^{m-1}}{m} \left( \sum_{j=1}^s (-x_i t)^j \right)^m\\
		& \quad = \sum_{m=1}^\infty \frac{(-1)^{m}}{m} \sum_{i=1}^{n} \sum_{k=m}^{sm} \sum_{\substack{t_1+t_2+\cdots+t_s=m\\t_2+2t_2+\cdots st_s = k}}
		\binom{t_1+t_2+\cdots+t_s}{t_1,t_2,\ldots,t_s} (-x_i t)^k\\
		& \quad = \sum_{m=1}^\infty \frac{(-1)^{m}}{m} \sum_{k=m}^{sm} \sum_{\substack{t_1+t_2+\cdots+t_s=m\\t_2+2t_2+\cdots st_s = k}}
		\binom{t_1+t_2+\cdots+t_s}{t_1,t_2,\ldots,t_s} p_k(x_1,x_2,\ldots,x_n) (-t)^k\\
		& \quad =\sum_{k=1}^\infty \sum_{t_1+2t_2+\cdots+st_s=k} \frac{(-1)^{k+t_1+t_2+\cdots+t_s}}{t_1+t_2+\cdots+t_s}
		\binom{t_1+t_2+\cdots+t_s}{t_1,t_2,\ldots,t_s} p_k(x_1,x_2,\ldots,x_n) t^k
		\end{align*}
		and
		\begin{align}
		& \ln\left(1+\sum_{k=1}^\infty H_k^{(s)}(x_1,x_2,\ldots,x_n) t^k \right)\label{Eq2.11} \\
		& = \sum_{k=1}^\infty \sum_{t_1+2t_2+\cdots+kt_k=k}
	    \frac{(-1)^{1+t_1+t_2+\cdots+t_k}}{t_1+t_2+\cdots+t_k}
		\binom{t_1+t_2+\cdots+t_k}{t_1,t_2,\ldots,t_k}
		\prod_{i=1}^k \left( H_i^{(s)} \right)^{t_i} t^k.\notag
		\end{align}
		The proof is finished.
	}	
\end{proof}

\begin{theorem}\label{T2.6}
	Let $k$, $n$ and $s$ be three positive integers and let $x_1,x_2,\ldots,x_n$ be independent
	variables. The generalized symmetric functions $E_k^{(s)}=E_k^{(s)}(x_1,x_2,\ldots,x_n)$, $H_k^{(s)}=H_k^{(s)}(x_1,x_2,\ldots,x_n)$ and $P_k^{(s)}=P_k^{(s)}(x_1,x_2,\ldots,x_n)$ are related by
	$$
	P_k^{(s)} = \sum_{t_1+2t_2+\ldots+kt_k=k} \frac{(-1)^{1+t_1+t_2+\cdots+t_k}\cdot k}{t_1+t_2+\cdots+t_k} \binom{t_1+t_2+\cdots+t_k}{t_1,t_2,\ldots,t_k} \prod_{i=1}^k \left(H_i^{(s)}\right)^{t_i}
	$$
	and
	$$
	P_k^{(s)} = \sum_{t_1+2t_2+\ldots+kt_k=k} \frac{(-1)^{k+t_1+t_2+\cdots+t_k}\cdot k}{t_1+t_2+\cdots+t_k} \binom{t_1+t_2+\cdots+t_k}{t_1,t_2,\ldots,t_k} \prod_{i=1}^k \left(E_i^{(s)}\right)^{t_i}.
	$$
\end{theorem}

\begin{proof}
	According to \eqref{Eq2.8}, \eqref{Eq2.10} and \eqref{Eq2.11}, we have
	\begin{align*}
	& \sum_{k=1}^\infty P_k^{(s)} (x_1,x_2,\ldots,x_n) t^{k-1}\\
	& \quad = \frac{d}{dt} \ln \prod_{i=1}^{n}\big(1-x_{i}t+\cdots +(-x_it)^s\big)^{-1} \\
	& \quad = \sum_{k=1}^\infty \sum_{t_1+2t_2+\cdots+kt_k=k}
	\frac{(-1)^{1+t_1+t_2+\cdots+t_k}\cdot k}{t_1+t_2+\cdots+t_k}
	\binom{t_1+t_2+\cdots+t_k}{t_1,t_2,\ldots,t_k}
	\prod_{i=1}^k \left( H_i^{(s)} \right)^{t_i} t^{k-1}
	\end{align*}
and
	\begin{align*}
	& \sum_{k=1}^\infty (-1)^{k-1}P_k^{(s)} (x_1,x_2,\ldots,x_n) t^{k-1}\\
	& \quad = \frac{d}{dt} \ln\left(\sum_{k=0}^\infty E_k^{(s)}(x_1,x_2,\ldots,x_n) t^k \right)\\
	& \quad = \sum_{k=1}^\infty \sum_{t_1+2t_2+\cdots+kt_k=k}
	\frac{(-1)^{1+t_1+t_2+\cdots+t_k}\cdot k}{t_1+t_2+\cdots+t_k}
	\binom{t_1+t_2+\cdots+t_k}{t_1,t_2,\ldots,t_k}
	\prod_{i=1}^k \left( E_i^{(s)} \right)^{t_i} t^{k-1}.
	\end{align*}
	These conclude the proof.
\end{proof}

As a consequence of Theorems \ref{T2.5} and \ref{T2.6}, we remark the following family of identities.

\begin{corollary}\label{C2.7}
	Let $k$ and $s$ be two positive integers. Then
	$$
	\sum_{t_1+2t_2+\cdots+st_s=k}
	\frac{(-1)^{t_1+t_2+\cdots+t_s}\cdot k}{t_1+t_2+\cdots+t_s}
	\binom{t_1+t_2+\cdots+t_s}{t_1,t_2,\ldots,t_s}
	= \begin{cases}
	s,& \text{if $k\equiv 0 \bmod {(s+1)}$,}\\
	-1, & \text{otherwise.}
	\end{cases}
	$$
\end{corollary}

It is well-known that the complete and elementary symmetric functions can be expressed in terms of the power sum symmetric functions, i.e.,
\begin{equation}\label{Eqp1}
h_k = \sum_{t_1+2t_2+\ldots+kt_k=k} \frac{1}{1^{t_1}t_1! 2^{t_2}t_2! \cdots k^{t_k}t_k!}
p_1^{t_1}p_2^{t_2}\cdots p_k^{t_k}
\end{equation}
and
\begin{equation}\label{Eqp2}
e_k = \sum_{t_1+2t_2+\ldots+kt_k=k} \frac{(-1)^{k+t_1+t_2+\cdots+t_k}}{1^{t_1}t_1! 2^{t_2}t_2! \cdots k^{t_k}t_k!}  p_1^{t_1}p_2^{t_2}\cdots p_k^{t_k}.
\end{equation}

The following result provides a generalization of these relations.
	
\begin{theorem}\label{T2.7}
	Let $k$, $n$ and $s$ be three positive integers and let $x_1,x_2,\ldots,x_n$ be independent
	variables. The generalized symmetric functions $E_k^{(s)}=E_k^{(s)}(x_1,x_2,\ldots,x_n)$, $H_k^{(s)}=H_k^{(s)}(x_1,x_2,\ldots,x_n)$ and $P_k^{(s)}=P_k^{(s)}(x_1,x_2,\ldots,x_n)$ are related by
	$$
	H_k^{(s)} = \sum_{t_1+2t_2+\ldots+kt_k=k} \frac{1}{1^{t_1}t_1! 2^{t_2}t_2! \cdots k^{t_k}t_k!}  \prod_{i=1}^k \left(P_i^{(s)}\right)^{t_i}
	$$
	and
	$$
	E_k^{(s)} = \sum_{t_1+2t_2+\ldots+kt_k=k} \frac{(-1)^{k+t_1+t_2+\cdots+t_k}}{1^{t_1}t_1! 2^{t_2}t_2! \cdots k^{t_k}t_k!}  \prod_{i=1}^k \left(P_i^{(s)}\right)^{t_i}.
	$$
\end{theorem}

\begin{proof}
	In order to prove this identity, we take into account the following two relations:
	$$
	\ln \left( \sum_{k=0}^\infty H_k^{(s)} t^k \right) = \sum_{k=1}^\infty \frac{t^k}{k} P_k^{(s)}
	\qquad\text{and}\qquad
	\ln \left( \sum_{k=0}^\infty E_k^{(s)} t^k \right) = \sum_{k=1}^\infty (-1)^{k-1} \frac{t^k}{k} P_k^{(s)}
	$$	
	Considering the exponential series
	$$\exp(z) = \sum_{k=0}^\infty \frac{z^k}{k!},\qquad |z|<1,$$
	we can write
	\begin{align*}
	& \sum_{k=0}^\infty H_k^{(s)} t^k \\
	& \quad = \exp\left( \sum_{k=1}^\infty \frac{t^k}{k!} P_k^{(s)}\right) \\
	& \quad = \sum_{m=0}^\infty \frac{1}{m!} \left( \sum_{k=1}^\infty \frac{t^k}{k!} P_k^{(s)} \right)^m\\
	& \quad = \sum_{m=0}^\infty \frac{1}{m!} \sum_{k=1}^\infty
	\sum_{\substack{t_1+t_2+\cdots t_k=m \\ t_1+2t_2+\cdots+kt_k=k}}
	\binom{t_1+t_2+\cdots+t_k}{t_1+t_2+\cdots+t_k} \prod_{i=1}^{k} \left( \frac{t^i}{i} P_i^{(s)} \right)^{t_i}\\
	& \quad = \sum_{k=1}^\infty \sum_{t_1+2t_2+\cdots+kt_k=k}
	\frac{1}{(t_1+t_2+\cdots+t_k)!} \binom{t_1+t_2+\cdots+t_k}{t_1+t_2+\cdots+t_k}
	\prod_{i=1}^{k} \left( \frac{1}{i} P_i^{(s)} \right)^{t_i} t^k\\
	& \quad = \sum_{k=1}^\infty \sum_{t_1+2t_2+\cdots+kt_k=k}
	\frac{1}{1^{t_1}t_1! 2^{t_2}t_2! \cdots k^{t_k}t_k!}
	\prod_{i=1}^k \left(P_i^{(s)}\right)^{t_i} t^k
	\end{align*}
	and
	\begin{align*}
	& \sum_{k=0}^\infty E_k^{(s)} t^k \\
	& \quad = \exp\left( \sum_{k=1}^\infty (-1)^{k-1} \frac{t^k}{k!} P_k^{(s)}\right) \\
	& \quad = \sum_{m=0}^\infty \frac{1}{m!}
	\left( \sum_{k=1}^\infty (-1)^{k-1} \frac{t^k}{k!} P_k^{(s)} \right)^m\\
	& \quad = \sum_{m=0}^\infty \frac{1}{m!} \sum_{k=1}^\infty
	\sum_{\substack{t_1+t_2+\cdots t_k=m \\ t_1+2t_2+\cdots+kt_k=k}}
	\binom{t_1+t_2+\cdots+t_k}{t_1+t_2+\cdots+t_k}
	\prod_{i=1}^{k} \left( (-1)^{i-1} \frac{t^i}{i} P_i^{(s)} \right)^{t_i}\\
	& \quad = \sum_{k=1}^\infty \sum_{t_1+2t_2+\cdots+kt_k=k}
	\frac{(-1)^{k+t_1+t_2+\cdots+t_k}}{1^{t_1}t_1! 2^{t_2}t_2! \cdots k^{t_k}t_k!}
	\prod_{i=1}^k \left(P_i^{(s)}\right)^{t_i} t^k.
	\end{align*}
	Thus we arrive at our identities.
\end{proof}

At the end of this section, we remark the following recurrence relations for the generalized symmetric functions $E_k^{(s)}$ and $H_k^{(s)}$.

\begin{theorem}
	Let $k$, $n$ and $s$ be three positive integers and let $x_1,x_2,\ldots,x_n$ be independent
	variables. Then
	\begin{align*}
	H_k^{(s)}(x_1,x_2,\ldots,x_n)
	& =(-x_n)^{s+1} H_{k-s-1}^{(s)}(x_1,x_2,\ldots,x_n) \\
	& \qquad + H_k^{(s)}(x_1,x_2,\ldots,x_{n-1})
	+ x_nH_{k-1}^{(s)}(x_1,x_2,\ldots,x_{n-1})
	\end{align*}
	and
	\begin{align*}
	E_k^{(s)}(x_1,x_2,\ldots,x_n)
	& =x_n E_{k-1}^{(s)}(x_1,x_2,\ldots,x_n) \\
	& \quad + E_k^{(s)}(x_1,x_2,\ldots,x_{n-1})
	-x_n^{s+1} E_{k-s-1}^{(s)}(x_1,x_2,\ldots,x_{n-1}).
	\end{align*}
\end{theorem}

\begin{proof}
	Taking into account \eqref{Eq1}, we can write
	\begin{align*}
	& \sum_{k=0}^\infty H_k^{(s)}(x_1,x_2,\ldots,x_n) t^k \\
	& \qquad = \frac{1}{1-x_nt+\cdots+(-x_nt)^s} \sum_{k=0}^\infty H_k^{(s)}(x_1,x_2,\ldots,x_{n-1}) t^k\\
	& \qquad = \frac{1+x_nt}{1-(-x_nt)^{s+1}} \sum_{k=0}^\infty H_k^{(s)}(x_1,x_2,\ldots,x_{n-1}) t^k.
	\end{align*}
	Thus we deduce that
	$$\big(1-(-x_nt)^{s+1}\big) \sum_{k=0}^\infty H_k^{(s)}(x_1,x_2,\ldots,x_n) t^k
	= (1+x_nt) \sum_{k=0}^\infty H_k^{(s)}(x_1,x_2,\ldots,x_{n-1}) t^k.$$
	Equating coefficients of $t^n$ on each side of this identity gives the first identity.
	The second identity follows in a similar way considering \eqref{Eq2}, so we omit the details.
\end{proof}

\section{Generalized symmetric functions in terms of the complete and elementary symmetric functions}
\label{S3}

It is well known that every symmetric function can be expressed as a sum of homogeneous symmetric functions. The homogeneous symmetric functions of degree $k$ in $n$ variables form a vector space, denoted $\varLambda^k_n$. There are several important bases for $\varLambda^k_n$, which are indexed by integer partitions of $k$. Proofs and details about these facts can be found in Macdonald's book \cite{Mck}. In this section, we express the generalized symmetric functions $H_k^{(s)}=H_k^{(s)}(x_{1},x_{2},\ldots,x_{n})$ and $E_k^{(s)}=E_k^{(s)}(x_{1},x_{2},\ldots,x_{n})$ in terms of the complete and elementary symmetric functions. To do this,
for each partition $\lambda$ we note
$$f_{\lambda}(x_1,x_2,\ldots,x_n)=\prod_{i=1}^{\ell(\lambda)} f_{\lambda_i}(x_1,x_2,\ldots,x_n),$$
where $f$ is any of these complete or elementary symmetric functions.

\begin{theorem}\label{T3.1}
	Let $k$ and $s$ be two positive integers. Then
		$$ H_{k}^{(s)}
			=(-1)^{k} \sum_{\substack{ \lambda \vdash k \\ l(\lambda) \leq s}} m_{\lambda}(\omega_{1,s+1},\omega_{2,s+1},\ldots, \omega_{s,s+1})h_{\lambda} $$
	and
		$$ E_{k}^{(s)}
			=(-1)^{k} \sum_{\substack{ \lambda \vdash k \\ l(\lambda) \leq s}} m_{\lambda}(\omega_{1,s+1},\omega_{2,s+1},\ldots, \omega_{s,s+1})e_{\lambda},$$
	where $\omega_{j,s+1}=e^{2j \pi i/(s+1)}$ with $j=1,2,\ldots,s$.
\end{theorem}

\begin{proof}
	\allowdisplaybreaks{
	Taking into account the generating functions \eqref{e1} and \eqref{Eq1}, we can write
	\begin{align*}
	& \sum_{k=0}^\infty H_k^{(s)} (x_1,x_2,\ldots,x_n) t^k  \\
	& \qquad = \prod_{i=1}^n \big(1-x_it+\cdots+(-x_it)^{s}\big)^{-1} \\
	& \qquad = \prod_{i=1}^n \prod_{j=1}^{s} ( 1+\omega_{j,s+1} x_i t)^{-1} = \prod_{j=1}^{s} \prod_{i=1}^n ( 1+\omega_{j,s+1} x_i t)^{-1}\\
	& \qquad = \prod_{j=1}^{s} \sum_{k=0}^\infty (-\omega_{j,s+1})^k h_k(x_1,x_2,\ldots,x_n) t^k\\
	& \qquad = \sum_{k=0}^\infty \left(\sum_{\substack{j_1+j_2+\cdots+j_{s}=k\\j_i\ge 0}}
	(-1)^{j_1+j_2+\cdots+j_{s}} \omega_{1,s+1}^{j_1}\omega_{2,s+1}^{j_2}\cdots \omega_{s,s+1}^{j_{s}} h_{j_1} h_{j_2} \ldots h_{j_{s}} \right) t^k \\
	& \qquad = \sum_{k=0}^\infty  \sum_{\substack{\lambda\vdash k\\\ell(\lambda)\leq s}} (-1)^k m_{\lambda}\left(\omega_{1,s+1},\omega_{2,s+1},\ldots,\omega_{s,s+1}\right)  h_{\lambda}(x_1,x_2,\ldots,x_n) t^k.
	\end{align*}
	The second identity follows in a similar way, considering the generating functions \eqref{e2} and \eqref{Eq2}. We have
	\begin{align*}
	& \sum_{k=0}^\infty E_k^{(s)} (x_1,x_2,\ldots,x_n) t^k  \\
	& \qquad = \prod_{i=1}^n \big(1+x_it+\cdots+(x_it)^{s}\big) \\
	& \qquad = \prod_{j=1}^{s} \prod_{i=1}^n ( 1-\omega_{j,s+1} x_i t)\\
	& \qquad = \prod_{j=1}^{s} \sum_{k=0}^\infty (-\omega_{j,s+1})^k e_k(x_1,x_2,\ldots,x_n) t^k\\
	& \qquad = \sum_{k=0}^\infty \left(\sum_{\substack{j_1+j_2+\cdots+j_{s}=k\\j_i\ge 0}}
	(-1)^{j_1+j_2+\cdots+j_{s}} \omega_{1,s+1}^{j_1}\omega_{2,s+1}^{j_2}\cdots \omega_{s,s+1}^{j_{s}} e_{j_1} e_{j_2} \ldots e_{j_{s}} \right) t^k \\
	& \qquad = \sum_{k=0}^\infty  \sum_{\substack{\lambda\vdash k\\\ell(\lambda)\leq s}} (-1)^k m_{\lambda}\left(\omega_{1,s+1},\omega_{2,s+1},\ldots,\omega_{s,s+1}\right)  e_{\lambda}(x_1,x_2,\ldots,x_n) t^k
	\end{align*}
	and the proof is finished.
}
\end{proof}

The Ferrers diagram of a partition $[\lambda_1,\lambda_2,\ldots,\lambda_k]$ is the $k$-row left-justified array of dots with $\lambda_i$ dots in the $i$-th row.
The conjugate of a partition into $s$ parts, obtained by transposing the Ferrers diagram, is a partition with largest part $s$ and vice versa. The action of conjugation establishes a $1–1$ correspondence between partitions into $s$ parts and partitions with largest part $s$. Considering \eqref{Eq3} and Theorem \ref{T3.1}, we obtain a surprising identity involving
this $1–1$ correspondence between partitions into $s$ parts and partitions with largest part $s$.

\begin{corollary}\label{C3.2}
	Let $k$, $n$ and $s$ be three positive integers and let $x_1,x_2,\ldots,x_n$ be independent
	variables. Then
	$$
	\sum_{\substack{\lambda \vdash k \\ \lambda_1\leq s}}
	m_{\lambda}(x_1,x_2,\ldots,x_n)
	= (-1)^{k} \sum_{\substack{ \lambda \vdash k \\ l(\lambda) \leq  s}} m_{\lambda}(\omega_{1,s+1},\omega_{2,s+1},\ldots,\omega_{s,s+1})
	e_{\lambda}(x_{1},x_{2},\ldots,x_{n}).
	$$
	where $\omega_{j,s+1}=e^{2j \pi i/(s+1)}$ with $j=1,2,\ldots,s$.
\end{corollary}
The following result allows us to express the generalized symmetric function $H_k^{(s)}$ and $E_k^{(s)}$ as convolutions involving the complete and elementary symmetric functions.

\begin{theorem}\label{T3.3}
	Let $k$, $n$ and $s$ be three positive integers and let $x_1,x_2,\ldots,x_n$ be independent
	variables. Then
	$$H_{k}^{(s-1)}(x_{1},x_{2},\ldots,x_{n})
	=\sum_{j=0}^{\lfloor k/s \rfloor} (-1)^{sj} h_{j}(x_{1}^{s},x_{2}^{s},\ldots,x_{n}^{s})
	e_{k-sj}(x_{1},x_{2},\ldots,x_{n})$$
	and
     $$E_{k}^{(s-1)}(x_{1},x_{2},\ldots,x_{n})
     =\sum_{j=0}^{\lfloor k/s \rfloor} (-1)^{j} e_{j}(x_{1}^{s},x_{2}^{s},\ldots,x_{n}^{s})
		h_{k-sj}(x_{1},x_{2},\ldots,x_{n}).$$
\end{theorem}

\begin{proof}
\allowdisplaybreaks{
According to \eqref{Eq1}, we have
\begin{align*}
& \sum_{k=0}^\infty H_k^{(s-1)} (x_{1},x_{2},\ldots,x_{n}) t^k\\
& \qquad =\prod_{i=1}^{n} \big(1-x_{i}t+\cdots+(-x_{i}t)^{s-1}\big)^{-1} \\
& \qquad=\left(\prod_{i=1}^{n}\frac{1}{1-(-x_{i}t)^{s}}\right)\left( \prod_{i=1}^{n} (1+x_{i}t)\right)\\
& \qquad=\left(\sum_{j=0}^\infty  h_{j}(x_{1}^{s},x_{2}^{s},\ldots,x_{n}^{s})(-t)^{sj} \right)\left(\sum_{j=0}^\infty e_{j}(x_{1},x_{2},\ldots,x_{n})t^{j}\right)\\
& \qquad=\sum_{k=0}^\infty \left(\sum_{j=0}^{\lfloor k/s \rfloor}(-1)^{sj}h_{j}(x_{1}^{s},x_{2}^{s},\ldots,x_{n}^{s})
e_{k-sj}(x_{1},x_{2},\ldots,x_{n})\right)t^{k}
\end{align*}
and
\begin{align*}
& \sum_{k=0}^\infty E_{k}^{(s-1)}(x_{1},x_{2},\ldots,x_{n})t^{k}\\ &\quad=\left(\prod_{i=1}^{n}(1-x^{s}_{i}t^{s})\right)\left(\prod_{i=1}^{n}\frac{1}{1-x_{i}t}\right)\\
&\quad=\left(\sum_{j=0}^\infty (-1)^{j} e_{j}(x_{1}^{s},x_{2}^{s},\ldots,x_{n}^{s}) t^{sj} \right) \left(\sum_{j=0}^\infty h_{j}(x_{1},x_{2},\ldots,x_{n})t^{j}\right)\\
&\quad=\sum_{k=0}^\infty \left( \sum_{j=0}^{\lfloor k/s \rfloor} (-1)^{j} e_{j}(x_{1}^{s},x_{2}^{s},\ldots,x_{n}^{s})
h_{k-sj}(x_{1},x_{2},\ldots,x_{n})\right)t^{k}.
\end{align*}
As required.}
\end{proof}

\begin{corollary}\label{C3.4}
	Let $k$, $n$ and $s$ be three positive integers and let $x_1,x_2,\ldots,x_n$ be independent
	variables. Then
	\begin{align*}
	& \sum_{j=0}^{\lfloor k/s \rfloor} (-1)^{sj}h_{j}(x_{1}^{s},x_{2}^{s},\ldots,x_{n}^{s})
	e_{k-sj}(x_{1},x_{2},\ldots,x_{n})\\
	& \qquad\qquad =  (-1)^{k} \sum_{\substack{ \lambda \vdash k \\ l(\lambda) < s}} m_{\lambda}(\omega_{1,s},\omega_{2,s},\ldots, \omega_{s-1,s}) h_{\lambda}(x_{1},x_{2},\ldots,x_{n})
	\end{align*}
	and
	\begin{align*}
	& \sum_{j=0}^{\lfloor k/s \rfloor} (-1)^{j}e_{j}(x_{1}^{s},x_{2}^{s},\ldots,x_{n}^{s})
	h_{k-sj}(x_{1},x_{2},\ldots,x_{n})\\
	& \qquad\qquad =  (-1)^{k} \sum_{\substack{ \lambda \vdash k \\ l(\lambda) < s}} m_{\lambda}(\omega_{1,s},\omega_{2,s},\ldots, \omega_{s-1,s}) e_{\lambda}(x_{1},x_{2},\ldots,x_{n}).
	\end{align*}
\end{corollary}

We remark that the second identity of this corollary is known and can be seen in a recent paper of Merca \cite[Theorem 1.1]{Mer5}.

Now,  we are able to prove some formulas for the monomial symmetric function
$$m_{\lambda}(e^{2 \pi i/(s+1)},e^{4\pi i/(s+1)},\ldots,e^{2s \pi i/(s+1)}),$$
when $\lambda$ is a partition of $k$, $s\geq k-2$ and $\ell(\lambda)\leq s$.

\begin{corollary}
	Let $k$ be a positive integer and let $\lambda=[1^{t_1}2^{t_2}\ldots k^{t_k}]$ be a partition of $k$. Then
$$ 	m_{\lambda}(\omega_{1,k+1},\omega_{2,k+1},\ldots,\omega_{k,k+1}) = (-1)^{t_1+t_2+\cdots+t_k} \binom{t_1+t_2+\cdots+t_k}{t_1,t_2,\ldots,t_k}, $$
	where $\omega_{j,k+1}=e^{2j \pi i/(k+1)}$ with $j=1,2,\ldots,k$.
\end{corollary}

\begin{proof}
	The case $s=k+1$ of Theorem \ref{T3.3} reads as $$H_k^{(k)}=e_k.$$
	By Theorem \ref{T3.1}, we deduce that
	$$ e_k =  \sum_{\lambda \vdash k} (-1)^k m_{\lambda} (\omega_{1,k+1},\omega_{2,k+1},\ldots,\omega_{k,k+1}) h_\lambda.
	$$
	On the other hand, the relation
	$$ e_k = \sum_{\lambda \vdash k} (-1)^{k+\ell(\lambda)} \binom{\ell(\lambda)}{t_1,t_2,\ldots,t_k} h_\lambda
	$$	
	can be found in \cite[pp. 3-4]{Mac}. It is clear that
	\begin{align*}
	& \sum_{\lambda \vdash k} m_{\lambda} (\omega_{1,k+1},\omega_{2,k+1},\ldots,\omega_{k,k+1}) h_\lambda= \sum_{\lambda \vdash k} (-1)^{\ell(\lambda)} \binom{\ell(\lambda)}{t_1,t_2,\ldots,t_k} h_\lambda.
	\end{align*}
	The assertion of the corollary now follows by comparing coefficients of
	$h_\lambda$ on both sides of this equation.
\end{proof}

\begin{corollary}
	Let $k>1$ be a positive integer and let $\lambda=[1^{t_1}2^{t_2}\ldots k^{t_k}]$ be a partition of $k$ with $\ell(\lambda)<k$. Then
	$$ 	m_{\lambda}(\omega_{1,k},\omega_{2,k},\ldots,\omega_{k-1,k})
	= (-1)^{\ell(\lambda)}\left(1-\frac{k}{\ell(\lambda)} \right)
	\binom{\ell(\lambda)}{t_1,t_2,\ldots,t_k}, $$
	where $\omega_{j,k}=e^{2j \pi i/k}$ with $j=1,2,\ldots,k-1$.
\end{corollary}

\begin{proof}
	The case $s=k$ of Theorem \ref{T3.3} reads as $$H_k^{(k-1)}=e_k+(-1)^{k} p_k.$$
	By Theorem \ref{T3.1}, we deduce that
	$$
	e_k = \sum_{\substack{\lambda \vdash k\\ \ell(\lambda)<k}} (-1)^k m_{\lambda} (\omega_{1,k},\omega_{2,k},\ldots,\omega_{k-1,k}) h_{\lambda} - (-1)^k p_k.
	$$
	On the other hand, we have
	$$
	e_k = \sum_{\lambda \vdash k} (-1)^{k+\ell(\lambda)} \binom{\ell(\lambda)}{t_1,t_2,\ldots,t_k} h_{\lambda}
	$$	
    and
    $$
    p_k = \sum_{\lambda \vdash k} \frac{(-1)^{1+\ell(\lambda)}\cdot k}{\ell(\lambda)} \binom{\ell(\lambda)}{t_1,t_2,\ldots,t_k} h_{\lambda}.
    $$
	We can write
	\begin{align*}
	& \sum_{\substack{\lambda \vdash k\\ \ell(\lambda)<k}} (-1)^k m_{\lambda} (\omega_{1,k},\omega_{2,k},\ldots,\omega_{k-1,k}) h_{\lambda} \\
	& \qquad = \sum_{\lambda \vdash k} (-1)^{k+\ell(\lambda)} \binom{\ell(\lambda)}{t_1,t_2,\ldots,t_k} h_{\lambda}
	-\sum_{\lambda \vdash k} \frac{(-1)^{k+\ell(\lambda)}\cdot k}{\ell(\lambda)} \binom{\ell(\lambda)}{t_1,t_2,\ldots,t_k} h_{\lambda}\\
	& \qquad = \sum_{\lambda \vdash k} (-1)^{k+\ell(\lambda)}\left(1-\frac{k}{\ell(\lambda)} \right)
	\binom{\ell(\lambda)}{t_1,t_2,\ldots,t_k} h_{\lambda}
	\end{align*}
	and the proof is finished.
\end{proof}

\begin{corollary}
	Let $k>2$ be a positive integer and let $\lambda=[1^{t_1}2^{t_2}\ldots k^{t_k}]$ be a partition of $k$ with $\ell(\lambda)\leq k-2$. Then
	$$ 	m_{\lambda}(\omega_{1,k-1},\omega_{2,k-1},\ldots,\omega_{k-2,k-1})
	= (-1)^{\ell(\lambda)}\left(1-\frac{t_1\cdot (k-1)}{\ell(\lambda)^2-\ell(\lambda)} \right)
	\binom{\ell(\lambda)}{t_1,t_2,\ldots,t_k}, $$
	where $\omega_{j,k-1}=e^{2j \pi i/(k-1)}$ with $j=1,2,\ldots,k-2$.
\end{corollary}

\begin{proof}
	By Theorem \ref{T3.1}, we deduce that
	$$
	H_k^{(k-2)} = \sum_{\substack{\lambda \vdash k\\ \ell(\lambda)\leq k-2}} (-1)^k m_{\lambda} (\omega_{1,k-1},\omega_{2,k-1},\ldots,\omega_{k-2,k-1}) h_{\lambda}.
	$$	
	On the other hand, taking into account the case $s=k-1$ of Theorem \ref{T3.3}, we can write
	\begin{align*}
	H_k^{(k-2)} & =e_k+(-1)^{k-1} p_{k-1}h_1\\
	& = (-1)^k \sum_{t_1+2t_2+\cdots+kt_k=k} (-1)^{t_1+t_2+\cdots+t_k}
	\binom{t_1+t_2+\cdots+t_k}{t_1,t_2,\ldots,t_k} h_1^{t_1} h_2^{t_2} \cdots h_k^{t_k}\\
	& \qquad + (-1)^{k} \sum_{t_1+2t_2+\cdots+(k-1)t_{k-1}=k-1}
	\frac{(-1)^{t_1+t_2+\cdots+t_{k-1}}(k-1)}{t_1+t_2+\cdots +t_{k-1}} \times\\
	& \qquad\qquad\qquad \times \binom{t_1+t_2+\cdots+t_{k-1}}{t_1,t_2,\ldots,t_{k-1}} h_1^{1+t_1} h_2^{t_2} \cdots h_{k-1}^{t_{k-1}}\\
	& = (-1)^k \sum_{t_1+2t_2+\cdots+kt_k=k} (-1)^{t_1+t_2+\cdots+t_k}
	\binom{t_1+t_2+\cdots+t_k}{t_1,t_2,\ldots,t_k} h_1^{t_1} h_2^{t_2} \cdots h_k^{t_k}\\
	& \qquad - (-1)^{k} \sum_{\substack{t_1+2t_2+\cdots+kt_{k}=k\\t_1>0}}
	\frac{(-1)^{t_1+t_2+\cdots+t_{k}}(k-1)}{t_1+t_2+\cdots +t_{k}-1} \times\\
	& \qquad\qquad\qquad \times \binom{t_1+t_2+\cdots+t_{k}-1}{t_1-1,t_2,\ldots,t_{k}} h_1^{t_1} h_2^{t_2} \cdots h_{k}^{t_{k}}\\
	& = (-1)^k \sum_{t_1+2t_2+\cdots+kt_k=k} (-1)^{t_1+t_2+\cdots+t_k}
	\binom{t_1+t_2+\cdots+t_k}{t_1,t_2,\ldots,t_k} h_1^{t_1} h_2^{t_2} \cdots h_k^{t_k}\\
	& \qquad - (-1)^{k} \sum_{\substack{t_1+2t_2+\cdots+kt_{k}=k\\t_1>0}}
	\frac{(-1)^{t_1+t_2+\cdots+t_{k}}\cdot t_1\cdot (k-1)}{(t_1+t_2+\cdots +t_{k}-1)(t_1+t_2+\cdots +t_{k})} \times\\
	& \qquad\qquad\qquad \times \binom{t_1+t_2+\cdots+t_{k}}{t_1,t_2,\ldots,t_{k}} h_1^{t_1} h_2^{t_2} \cdots h_{k}^{t_{k}}\\
	& = (-1)^k \sum_{\lambda\vdash k} (-1)^{\ell(\lambda)} \left(1- \frac{t_1\cdot (k-1)}{\big(\ell(\lambda)-1\big)\ell(\lambda)} \right) \binom{\ell(\lambda)}{t_1,t_2,\ldots,t_k} h_\lambda.
	\end{align*}
	This concludes the proof.
\end{proof}

Inspired by Theorem \ref{T3.3}, we provide the following result.

\begin{theorem}\label{T3.4}
	Let $k$, $n$ and $s$ be three positive integers and let $x_1,x_2,\ldots,x_n$ be independent
	variables. Then
	$$
	h_k(x_1^s,x_2^s,\ldots,x_n^s)
	= (-1)^{k(s+1)} \sum_{j=0}^{ks} (-1)^{j} h_j(x_1,x_2,\ldots,x_n) H_{ks-j}^{(s-1)}(x_1,x_2,\ldots,x_n).
	$$
    and
    $$
	e_k(x_1^s,x_2^s,\ldots,x_n^s)
	= (-1)^k \sum_{j=0}^{ks} (-1)^{j} e_j(x_1,x_2,\ldots,x_n) E_{ks-j}^{(s-1)}(x_1,x_2,\ldots,x_n).
	$$
	If $k$ is not congruent to  $0$ modulo $s$ then
	$$
	\sum_{j=0}^k (-1)^j h_j(x_1,x_2,\ldots,x_n) H_{k-j}^{(s-1)}(x_1,x_2,\ldots,x_n) = 0
	$$
	and
	$$
	\sum_{j=0}^k (-1)^j e_j(x_1,x_2,\ldots,x_n) E_{k-j}^{(s-1)}(x_1,x_2,\ldots,x_n) = 0.
	$$	
\end{theorem}

\begin{proof}
	We have
	$$\sum_{k=0}^\infty H_k^{(s-1)}(x_1,x_2,\ldots,x_n) t^k = \prod_{i=1}^n \frac{1+x_it}{1+(-x_it)^s}$$
	and
	$$\sum_{k=0}^\infty E_k^{(s-1)}(x_1,x_2,\ldots,x_n) t^k  = \prod_{i=1}^n \frac{1-(x_it)^s}{1-x_it}.$$
	These relations can be rewritten as
	$$\prod_{i=1}^n \frac{1}{1+x_it} \sum_{k=0}^\infty H_k^{(s-1)}(x_1,x_2,\ldots,x_n) t^k
	= \prod_{i=1}^n \frac{1}{1+(-x_it)^s}$$
	and
	$$\prod_{i=1}^n (1-x_it) \sum_{k=0}^\infty E_k^{(s-1)}(x_1,x_2,\ldots,x_n) t^k
	= \prod_{i=1}^n \big(1-(x_it)^s\big).$$
	Thus we deduce that
	\begin{align*}
	& \sum_{k=0}^\infty (-1)^{k(s+1)} h_k(x_1^s,x_2^s,\ldots,x_n^s) t^{ks}\\
	& \qquad = \left( \sum_{k=0}^\infty (-1)^k h_k(x_1,x_2,\ldots,x_n) t^k \right)
	\left( \sum_{k=0}^\infty H_k^{(s-1)}(x_1,x_2,\ldots,x_n) t^k \right)
	\end{align*}
	and
	\begin{align*}
	& \sum_{k=0}^\infty (-1)^{k} e_k(x_1^s,x_2^s,\ldots,x_n^s) t^{ks}\\
	& \qquad = \left( \sum_{k=0}^\infty (-1)^k e_k(x_1,x_2,\ldots,x_n) t^k \right)
	\left( \sum_{k=0}^\infty E_k^{(s-1)}(x_1,x_2,\ldots,x_n) t^k \right).
	\end{align*}
	The proof follows easily by comparing the coefficients of $t^{ks}$ on both sides of these equations.	
\end{proof}


\section{Combinatorial interpretations of the generalized symmetric functions}
\label{S4}

Bazeniar et \textit{al.} \cite{BAZ}
showed that the generalized symmetric function $E_k^{(s)}$ is interpreted as weight-generating function of the lattice paths between the points $u=(0, 0)$ and $v=(k,n-1)$ with at most $s$ vertices in the eastern direction. For example, the paths from $(0,0)$ to $(3,2)$ associated to
$$
E_{3}^{(2)}(x_1,x_2,x_3)=x_1^2x_2+x_1x_2^2+x_1^2x_3+x_1x_3^2+x_2^2x_3+x_2x_2^3+x_1x_2x_3
$$
can be seen in Figure \ref{Fig1}.

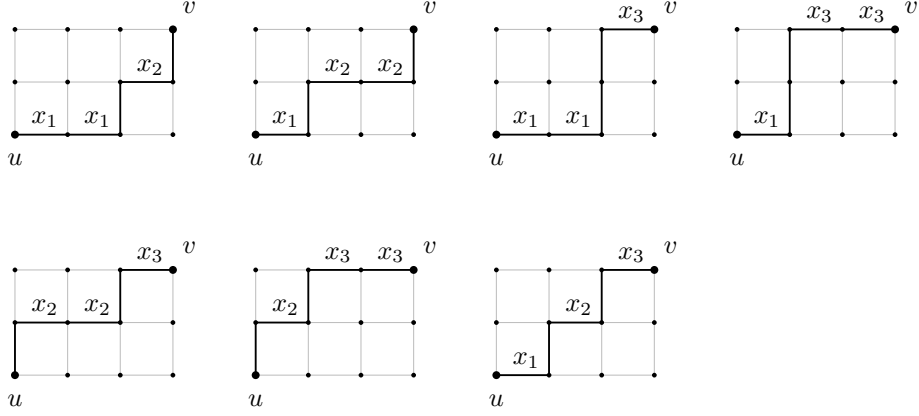
\begin{figure}[h]
\begin{center}
	\begin{tikzpicture}
	\draw[step=0.7cm,color=lightgray] (0,0) grid(2.1,1.4);
	\draw [line width=0.7pt](0,0) -- (0.7,0.0)-- (1.4,0.0) -- (1.4,0.7) -- (2.1,0.7) -- (2.1,1.4);
	
	\fill[black] (0,0) circle(1.5pt) ;
	\fill[black] (2.1,1.4) circle(1.5pt) ;
	
	\fill[] (0.7,0) circle (0.9pt);\fill[] (1.4,0) circle (0.9pt);\fill[] (2.1,0) circle (0.9pt);
	\fill[] (0,0.7) circle (0.9pt);\fill[] (0.7,0.7) circle (0.9pt);\fill[] (1.4,0.7) circle (0.9pt);
	\fill[] (2.1,0.7) circle (0.9pt);
	\fill[] (0,1.4) circle (0.9pt);\fill[] (0.7,1.4) circle (0.9pt);\fill[] (1.4,1.4) circle (0.9pt);

    \node[rectangle] at (0.4,0.2) {$x_{1}$};
    \node[rectangle] at (1.1,0.2) {$x_{1}$};
    \node[rectangle] at (1.8,0.9) {$x_{2}$};
	
	\node[right=0.1pt] at (2.1,1.7){$v$};
	\node at (0,-1/3){$u$};
	
\begin{scope}[xshift=1.0cm]
\begin{scope}[xshift=2.2cm]

\draw[step=0.7cm,color=lightgray] (0,0) grid(2.1,1.4);
\draw [line width=0.7pt](0,0) -- (0.7,0.0) -- (0.7,0.7) -- (1.4,0.7) -- (2.1,0.7) -- (2.1,1.4);

	\fill[black] (0,0) circle(1.5pt) ;
	\fill[black] (2.1,1.4) circle(1.5pt) ;

	\fill[] (0.7,0) circle (0.9pt);\fill[] (1.4,0) circle (0.9pt);\fill[] (2.1,0) circle (0.9pt);
	\fill[] (0,0.7) circle (0.9pt);\fill[] (0.7,0.7) circle (0.9pt);\fill[] (1.4,0.7) circle (0.9pt);
	\fill[] (2.1,0.7) circle (0.9pt);
	\fill[] (0,1.4) circle (0.9pt);\fill[] (0.7,1.4) circle (0.9pt);\fill[] (1.4,1.4) circle (0.9pt);

	\node[rectangle] at (0.4,0.2) {$x_{1}$};
	\node[rectangle] at (1.1,0.9) {$x_{2}$};
	\node[rectangle] at (1.8,0.9) {$x_{2}$};

	\node[right=0.1pt] at (2.1,1.7){$v$};
	\node at (0,-1/3){$u$};
	
	\begin{scope}[xshift=1.0 cm]
	\begin{scope}[xshift=2.2 cm]
	\draw[step=0.7cm,color=lightgray] (0,0) grid(2.1,1.4);
	\draw [line width=0.7pt](0,0) -- (0.7,0.0)-- (1.4,0.0) -- (1.4,0.7) -- (1.4,1.4) -- (2.1,1.4);
	
	\fill[black] (0,0) circle(1.5pt) ;
	\fill[black] (2.1,1.4) circle(1.5pt) ;
	
	\fill[] (0.7,0) circle (0.9pt);\fill[] (1.4,0) circle (0.9pt);\fill[] (2.1,0) circle (0.9pt);
	\fill[] (0,0.7) circle (0.9pt);\fill[] (0.7,0.7) circle (0.9pt);\fill[] (1.4,0.7) circle (0.9pt);
	\fill[] (2.1,0.7) circle (0.9pt);
	\fill[] (0,1.4) circle (0.9pt);\fill[] (0.7,1.4) circle (0.9pt);\fill[] (1.4,1.4) circle (0.9pt);
	
	\node[rectangle] at (0.4,0.2) {$x_{1}$};
	\node[rectangle] at (1.1,0.2) {$x_{1}$};
	\node[rectangle] at (1.8,1.6) {$x_{3}$};
	
	\node[right=0.1pt] at (2.1,1.7){$v$};
	\node at (0,-1/3){$u$};

	\begin{scope}[xshift=1.0 cm]
	\begin{scope}[xshift=2.2 cm]
\draw[step=0.7cm,color=lightgray] (0,0) grid(2.1,1.4);
\draw [line width=0.7pt](0,0) -- (0.7,0.0)-- (0.7,0.7) -- (0.7,1.4) -- (1.4,1.4) -- (2.1,1.4);

\fill[black] (0,0) circle(1.5pt) ;
\fill[black] (2.1,1.4) circle(1.5pt) ;

\fill[] (0.7,0) circle (0.9pt);\fill[] (1.4,0) circle (0.9pt);\fill[] (2.1,0) circle (0.9pt);
\fill[] (0,0.7) circle (0.9pt);\fill[] (0.7,0.7) circle (0.9pt);\fill[] (1.4,0.7) circle (0.9pt);
\fill[] (2.1,0.7) circle (0.9pt);
\fill[] (0,1.4) circle (0.9pt);\fill[] (0.7,1.4) circle (0.9pt);\fill[] (1.4,1.4) circle (0.9pt);

\node[rectangle] at (0.4,0.2) {$x_{1}$};
\node[rectangle] at (1.1,1.6) {$x_{3}$};
\node[rectangle] at (1.8,1.6) {$x_{3}$};

\node[right=0.1pt] at (2.1,1.7){$v$};
\node at (0,-1/3){$u$};

	\end{scope}
	\end{scope}
	
	\end{scope}
	\end{scope}
	
\end{scope}
\end{scope}

		\begin{scope}[yshift=-1.0 cm]
\begin{scope}[yshift=-2.2 cm]
\draw[step=0.7cm,color=lightgray] (0,0) grid(2.1,1.4);
\draw [line width=0.7pt](0,0) -- (0.0,0.7)-- (0.7,0.7) -- (1.4,0.7) -- (1.4,1.4) -- (2.1,1.4);

\fill[black] (0,0) circle(1.5pt) ;
\fill[black] (2.1,1.4) circle(1.5pt) ;

\fill[] (0.7,0) circle (0.9pt);\fill[] (1.4,0) circle (0.9pt);\fill[] (2.1,0) circle (0.9pt);
\fill[] (0,0.7) circle (0.9pt);\fill[] (0.7,0.7) circle (0.9pt);\fill[] (1.4,0.7) circle (0.9pt);
\fill[] (2.1,0.7) circle (0.9pt);
\fill[] (0,1.4) circle (0.9pt);\fill[] (0.7,1.4) circle (0.9pt);\fill[] (1.4,1.4) circle (0.9pt);

\node[rectangle] at (0.4,0.9) {$x_{2}$};
\node[rectangle] at (1.1,0.9) {$x_{2}$};
\node[rectangle] at (1.8,1.6) {$x_{3}$};

\node[right=0.1pt] at (2.1,1.7){$v$};
\node at (0,-1/3){$u$};

		\begin{scope}[xshift=1.0 cm]
\begin{scope}[xshift=2.2 cm]
\draw[step=0.7cm,color=lightgray] (0,0) grid(2.1,1.4);
\draw [line width=0.7pt](0,0) -- (0.0,0.7)-- (0.7,0.7) -- (0.7,1.4) -- (1.4,1.4) -- (2.1,1.4);

\fill[black] (0,0) circle(1.5pt) ;
\fill[black] (2.1,1.4) circle(1.5pt) ;

\fill[] (0.7,0) circle (0.9pt);\fill[] (1.4,0) circle (0.9pt);\fill[] (2.1,0) circle (0.9pt);
\fill[] (0,0.7) circle (0.9pt);\fill[] (0.7,0.7) circle (0.9pt);\fill[] (1.4,0.7) circle (0.9pt);
\fill[] (2.1,0.7) circle (0.9pt);
\fill[] (0,1.4) circle (0.9pt);\fill[] (0.7,1.4) circle (0.9pt);\fill[] (1.4,1.4) circle (0.9pt);

\node[rectangle] at (0.4,0.9) {$x_{2}$};
\node[rectangle] at (1.1,1.6) {$x_{3}$};
\node[rectangle] at (1.8,1.6) {$x_{3}$};

\node[right=0.1pt] at (2.1,1.7){$v$};
\node at (0,-1/3){$u$};

		\begin{scope}[xshift=1.0 cm]
\begin{scope}[xshift=2.2 cm]
\draw[step=0.7cm,color=lightgray] (0,0) grid(2.1,1.4);
\draw [line width=0.7pt](0,0) -- (0.7,0.0)-- (0.7,0.7) -- (1.4,0.7) -- (1.4,1.4) -- (2.1,1.4);

\fill[black] (0,0) circle(1.5pt) ;
\fill[black] (2.1,1.4) circle(1.5pt) ;

\fill[] (0.7,0) circle (0.9pt);\fill[] (1.4,0) circle (0.9pt);\fill[] (2.1,0) circle (0.9pt);
\fill[] (0,0.7) circle (0.9pt);\fill[] (0.7,0.7) circle (0.9pt);\fill[] (1.4,0.7) circle (0.9pt);
\fill[] (2.1,0.7) circle (0.9pt);
\fill[] (0,1.4) circle (0.9pt);\fill[] (0.7,1.4) circle (0.9pt);\fill[] (1.4,1.4) circle (0.9pt);

\node[rectangle] at (0.4,0.2) {$x_{1}$};
\node[rectangle] at (1.1,0.9) {$x_{2}$};
\node[rectangle] at (1.8,1.6) {$x_{3}$};

\node[right=0.1pt] at (2.1,1.7){$v$};
\node at (0,-1/3){$u$};
\end{scope}
\end{scope}
\end{scope}
\end{scope}
\end{scope}
\end{scope}

	\end{tikzpicture}
\end{center}
\caption{The seven paths from $u$ to $v$ associated to $E_{3}^{(2)}(x_1,x_2,x_3)$. }
\label{Fig1}
\end{figure}

According to \cite[Theorem 3.2]{BAZ}, the number of lattice paths
from $(0, 0)$ to $(k,n-1)$ taking at most $s$ vertices in the eastern direction is exactly the bi$^s$nomial coefficient, i.e.,
$$
\binom{n}{k}_s = E_k^{(s)}(\underbrace{1,1,\ldots,1}_n).
$$
By Theorem \ref{T3.3}, we deduce that the bi$^s$nomial coefficient can be expressed in terms of the classical binomial coefficients, i.e.,
$$
\binom{n}{k}_{s-1} = \sum_{j=0}^{\lfloor k/s \rfloor} (-1)^j \binom{n}{j} \binom{n+k-sj-1}{k-sj}.
$$
We remark that this identity is given by Theorem 2.1 in \cite{bsb}. In addition, by Theorem \ref{T3.3} we obtain the following analogs of this identity.

\begin{corollary}
	Let $k,n$ and $s$ be three positive integers. Then
$$
{n \brack k}^{(s-1)}_{q} = \sum_{j=0}^{\lfloor k/s \rfloor} (-1)^j q^{s\binom{j}{2}} {n \brack j}_{q^s} {n+k-sj-1 \brack k-sj}_q,
$$
where
$$
{n \brack k}^{(s)}_{q} = E_k^{(s)}(1,q,\ldots,q^{n-1})
$$
is the $q$-bi$^s$nomial coefficient.
\end{corollary}

\begin{corollary}
	Let $k,n$ and $s$ be three positive integers. Then
$$
{n \brack k}^{(s-1)}_{p,q} = \sum_{j=0}^{\lfloor k/s \rfloor} (-1)^j p^{s\binom{n-j}{2}}q^{s\binom{j}{2}} {n \brack j}_{p^s, q^s} {n+k-sj-1 \brack k-sj}_{p,q},
$$
where
$$
{n \brack k}^{(s)}_{p,q} = E_k^{(s)}(p^{n-1},p^{n-2}q,\ldots,q^{n-1})
$$
is the $p,q$-bi$^s$nomial coefficient.
\end{corollary}

These expressions of the $q$-bi$^s$nomial (rep. $p,q$-bi$^s$nomial) coefficient in terms of  $q$-binomial  coefficients ${n \brack k}_{q}$ (resp. $p,q$-binomial coefficients ${n \brack k}_{p,q}$) seem to be new.

    And Theorem \ref{T3.4} allow us to express the binomial coefficient and its $q^s$-analogue in term of the bi$^s$niomial coefficient and its $q$-analogue, respectively.
    \begin{corollary} Let $k,n$ and $s$ be three positive integers. Then
    $$
    \binom{n}{k}=\sum_{j=0}^{ks} (-1)^{k+j} \binom{n}{j} \binom{n}{ks-j}_{s-1}
    $$
    and
    $$
    {n \brack k}_{q^s}=\sum_{j=0}^{ks} (-1)^{k+j} q^{\binom{j}{2}-s\binom{k}{2}} {n \brack j}_{q} {n \brack ks-j}^{(s-1)}_{q}.
    $$
    \end{corollary}

	Inspired by this interpretation of the generalized symmetric function $E_k^{(s)}$, we provide in this section a combinatorial interpretation for the generalized symmetric function $H_{k}^{(s)}$. To do this we consider the following result which allows us to express the generalized symmetric function $H_k^{(s)}$ in terms of the monomial symmetric functions $m_\lambda$ considering all the partitions of $k$ into parts congruent to $0$ or $1$ modulo $s+1$.
	
\begin{theorem}\label{T4.1}
	Let $k,n$ and $s$ be three positive integers and let $x_{1},x_{2},\ldots
	,x_{n}$ be independent variables. Then%
	\[
	H_{k}^{(s)}\left( x_{1},x_{2},\ldots ,x_{n}\right)
	= \sum\limits_{\substack{ \lambda \vdash k \\ \lambda _{i}\equiv \{0,1\} \bmod {(s+1)} }}
	(-1)^{k+\sum\limits_{i=1}^{\ell(\lambda)} \lambda _{i} \bmod {(s+1)}}
	m_{\lambda }(x_{1},x_{2},\ldots ,x_{n}).
	\]
\end{theorem}

\begin{proof}
	\allowdisplaybreaks{
	According to \eqref{Eq1}, we can write
	\begin{align*}
	&\sum_{k=0}^\infty H_k^{(s)} (x_1,x_2,\ldots,x_n) t^k \\
	& \qquad = \prod_{i=1}^n \big(1-x_it+\cdots +(-x_it)^s \big)^{-1}\\
	& \qquad = \prod_{i=1}^n \frac{1+x_i t}{1-(-x_it)^{s+1}}\\
	& \qquad = \prod_{i=1}^n \big(1-(-x_it)\big) \sum_{j=0}^\infty (-x_it)^{j(s+1)}\\
	& \qquad = \prod_{i=1}^n \left(  \sum_{j=0}^\infty (-x_it)^{j(s+1)}-\sum_{j=0}^\infty (-x_it)^{j(s+1)+1}\right) \\
	& \qquad = \sum_{k=0}^\infty \left( \sum\limits_{\substack{ \lambda \vdash k \\ \lambda _{i}\equiv \{0,1\} \bmod {(s+1)} }}
	(-1)^{\sum_{i=1}^{\ell(\lambda)} \lambda _{i} \bmod {(s+1)}}
	m_{\lambda }(x_{1},x_{2},\ldots ,x_{n}) \right) (-t)^k
	\end{align*}
	and the proof follows easily.}
\end{proof}

\begin{remark}
	When $s$ is odd, we have
	$$
	H_{k}^{(s)}\left( x_{1},x_{2},\ldots ,x_{n}\right)
	=\sum_{\substack{ \lambda \vdash k \\ \lambda _{i}\equiv \{0,1\} \bmod {(s+1)} }}
	m_{\lambda }(x_{1},x_{2},\ldots,x_{n}).
	$$	
\end{remark}

The following consequence of Theorem \ref{T4.1} is an analogy of Corollary \ref{C3.2} establishing a connection between all the partitions of $k$ into parts congruent to $0$ or $1$ modulo $s+1$ and the partitions of $k$ into at most $s$ parts.

\begin{corollary}\label{C4.3}
	Let $k$, $n$ and $s$ be three positive integers and let $x_1,x_2,\ldots,x_n$ be independent
	variables. Then
	\begin{align*}
	& \sum\limits_{\substack{ \lambda \vdash k \\ \lambda _{i}\equiv \{0,1\} \bmod {(s+1)} }}
	(-1)^{\sum\limits_{i=1}^{\ell(\lambda)} \lambda _{i} \bmod {(s+1)}}
	m_{\lambda }(x_{1},x_{2},\ldots ,x_{n})\\
	&\qquad\qquad\qquad =  \sum_{\substack{ \lambda \vdash k \\ l(\lambda) \leq s}} m_{\lambda}(\omega_{1,s+1},\omega_{2,s+1},\ldots, \omega_{s,s+1}) h_{\lambda}(x_{1},x_{2},\ldots,x_{n}).
	\end{align*}
\end{corollary}

Let $\mathcal{P}_{n,k}^s$ be the set of the lattice paths
between the points $u=(0,0)$ and $v=(k,n-1)$ where the number of the vertices in the eastern direction is congruent to $0$ or $1$ modulo $s+1$. For $P=(p_{1},p_2,\ldots ,p_{n+k-1})\in \mathcal{P}_{n,k}^s $, we consider
$$n_{i}(P):=\text{\ the number of the eastern step modulo }(s+1) \text{\ in level }i.$$
and the $H^{(s)}$-labeling which assigns the label for each eastern step as follows
$$
L\left(p_{i}\right):=\left( \text{the number of northern }p_{j}\text{ preceding }p_{i}\right) +1.
$$
Figure \ref{Esl} shows the $H^{(s)}$-labeling.
\begin{figure}[h]
\begin{center}
\begin{tikzpicture}
\draw[step=0.7cm,color=lightgray] (0,0) grid(2.1,1.4);
\draw [line width=0.7pt](0,0) -- (0,0.7)-- (0.7,0.7) -- (1.4,0.7) -- (1.4,0.7)--(2.1,0.7)--(2.1,1.4);

\fill[black] (0,0) circle(1.5pt) ;
\fill[black] (2.1,1.4) circle(1.5pt) ;

\fill[] (0.7,0) circle (0.9pt);\fill[] (1.4,0) circle (0.9pt);\fill[] (2.1,0) circle (0.9pt);
\fill[] (0,0.7) circle (0.9pt);\fill[] (0.7,0.7) circle (0.9pt);\fill[] (1.4,0.7) circle (0.9pt);
\fill[] (2.1,0.7) circle (0.9pt);
\fill[] (0,1.4) circle (0.9pt);\fill[] (0.7,1.4) circle (0.9pt);\fill[] (1.4,1.4) circle (0.9pt);

\node[right=0.1pt] at (2.1,1.7){$v$};
\node at (0,-1/3){$u$};
\node[rectangle] at (0.2,0.3) {$p_{1}$};
\node[rectangle] at (0.4,0.9) {$p_{2}$};
\node[rectangle] at (1.1,0.9) {$p_{3}$};
\node[rectangle] at (1.75,0.9) {$p_{4}$};
\node[rectangle] at (2.4,1.1) {$p_{5}$};

\begin{scope}[xshift=2.7cm]
\draw [->] (0.2,0.4)--(2.2,0.4);
\node[rectangle] at (1.2,0.7) {$H^{(2)}-labeling$};
\begin{scope}[xshift=2.8cm]

\draw[step=0.7cm,color=lightgray] (0,0) grid(2.1,1.4);
\draw [line width=0.7pt](0,0) -- (0,0.7)-- (0.7,0.7) -- (1.4,0.7) -- (1.4,0.7)--(2.1,0.7)--(2.1,1.4);

\fill[black] (0,0) circle(1.5pt) ;
\fill[black] (2.1,1.4) circle(1.5pt) ;

\fill[] (0.7,0) circle (0.9pt);\fill[] (1.4,0) circle (0.9pt);\fill[] (2.1,0) circle (0.9pt);
\fill[] (0,0.7) circle (0.9pt);\fill[] (0.7,0.7) circle (0.9pt);\fill[] (1.4,0.7) circle (0.9pt);
\fill[] (2.1,0.7) circle (0.9pt);
\fill[] (0,1.4) circle (0.9pt);\fill[] (0.7,1.4) circle (0.9pt);\fill[] (1.4,1.4) circle (0.9pt);

\node[right=0.1pt] at (2.1,1.7){$v$};
\node at (0,-1/3){$u$};
\node[rectangle] at (0.4,0.9) {$2$};
\node[rectangle] at (1.1,0.9) {$2$};
\node[rectangle] at (1.75,0.9) {$2$};

\begin{scope}[xshift=2.5cm]
\draw [->] (0.2,0.4)--(1.2,00.4);
\node[rectangle] at (0.7,0.6) {$x^P$};

\begin{scope}[xshift=1.6cm]

\draw[step=0.7cm,color=lightgray] (0,0) grid(2.1,1.4);
\draw [line width=0.7pt](0,0) -- (0,0.7)-- (0.7,0.7) -- (1.4,0.7) -- (1.4,0.7)--(2.1,0.7)--(2.1,1.4);

\fill[black] (0,0) circle(1.5pt) ;
\fill[black] (2.1,1.4) circle(1.5pt) ;

\fill[] (0.7,0) circle (0.9pt);\fill[] (1.4,0) circle (0.9pt);\fill[] (2.1,0) circle (0.9pt);
\fill[] (0,0.7) circle (0.9pt);\fill[] (0.7,0.7) circle (0.9pt);\fill[] (1.4,0.7) circle (0.9pt);
\fill[] (2.1,0.7) circle (0.9pt);
\fill[] (0,1.4) circle (0.9pt);\fill[] (0.7,1.4) circle (0.9pt);\fill[] (1.4,1.4) circle (0.9pt);

\node[right=0.1pt] at (2.1,1.7){$v$};
\node at (0,-1/3){$u$};

\node[rectangle] at (0.4,0.9) {$x_2$};
\node[rectangle] at (1.1,0.9) {$x_2$};
\node[rectangle] at (1.75,0.9) {$x_2$};

\end{scope}
\end{scope}
\end{scope}
\end{scope}
\end{tikzpicture}
\end{center}
\caption{Illustration of $x_{2}^{3}$ by $H^{(2)}$-labeling.}
\label{Esl}
\end{figure}
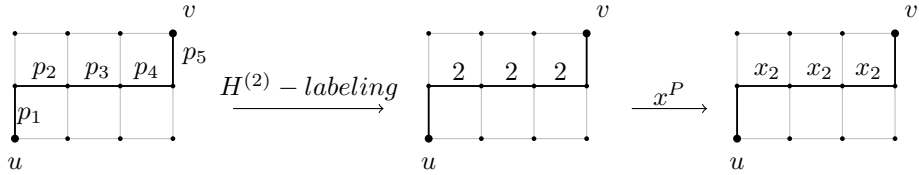

\begin{theorem} \label{T4.4}
	Let $k,n$ and $s$ be three positive integers and let $x_{1},x_{2},\ldots,x_{n}$ be independent variables. Then
\[
H_{k}^{(s)}\left( x_{1},x_2,\ldots ,x_{n}\right) =
\begin{cases}
\sum_{P\in \mathcal{P}_{n,k}^s}X^{P}, & \text{if }s\text{ odd,} \\
(-1)^{k}\sum_{P\in \mathcal{P}_{n,k}^s}(-1)^{P'} X^{P}, &
\text{otherwise}
\end{cases}
\]
with $X^{P}=\prod_i x_{L(p_{i})}$ and $P'=\sum_{i} n_{i}(P)$.
\end{theorem}

\begin{proof}
From Theorem \ref{T4.1}, is easy to see that the generalized symmetric function $H_{k}^{(s)}\left( x_{1},x_2,\ldots ,x_{n}\right)$ is a weight-generating function of lattice paths between two points. For each unit variable $x_{i}$ in this symmetric function we associate one unit horizontal (east) vertex, and if we suppose that each lattice path starting in $u=(0,0)$ then it ends in $v=(k,n-1)$ where the number of the vertices in the eastern direction equal to $0$ or $1$ modulo $(s+1)$.
\end{proof}

 Figure \ref{fg} shows the lattice path interpretation for
$$H_{3}^{(2)}(x_{1},x_{2},x_{3})=-x_{1}^{3}-x_{2}^{3}-x_{3}^{3}+x_1x_2x_3.$$

\begin{figure}[h]
	\begin{center}
		\begin{tikzpicture}
		\draw[step=0.7cm,color=lightgray] (0,0) grid(2.1,1.4);
		\draw [line width=0.7pt](0,0) -- (0.7,0)--(1.4,0) -- (2.1,0)--(2.1,1.4);
		
		\fill[black] (0,0) circle(1.5pt) ;
		\fill[black] (2.1,1.4) circle(1.5pt) ;
		
		\fill[] (0.7,0) circle (0.9pt);\fill[] (1.4,0) circle (0.9pt);\fill[] (2.1,0) circle (0.9pt);
		\fill[] (0,0.7) circle (0.9pt);\fill[] (0.7,0.7) circle (0.9pt);\fill[] (1.4,0.7) circle (0.9pt);
		\fill[] (2.1,0.7) circle (0.9pt);
		\fill[] (0,1.4) circle (0.9pt);\fill[] (0.7,1.4) circle (0.9pt);\fill[] (1.4,1.4) circle (0.9pt);

		\node[right=0.1pt] at (2.1,1.7){$v$};
		\node at (0,-1/3){$u$};
		\node[rectangle] at (0.4,0.3) {$x_{1}$};
		\node[rectangle] at (1.1,0.3) {$x_{1}$};
		\node[rectangle] at (1.8,0.3) {$x_{1}$};
		
		\begin{scope}[xshift=1.0cm]
		\begin{scope}[xshift=2.2cm]

		\draw[step=0.7cm,color=lightgray] (0,0) grid(2.1,1.4);
		\draw [line width=0.7pt](0,0) -- (0,0.7)-- (0.7,0.7) -- (1.4,0.7) -- (1.4,0.7)--(2.1,0.7)--(2.1,1.4);
		
		\fill[black] (0,0) circle(1.5pt) ;
		\fill[black] (2.1,1.4) circle(1.5pt) ;
		
		\fill[] (0.7,0) circle (0.9pt);\fill[] (1.4,0) circle (0.9pt);\fill[] (2.1,0) circle (0.9pt);
		\fill[] (0,0.7) circle (0.9pt);\fill[] (0.7,0.7) circle (0.9pt);\fill[] (1.4,0.7) circle (0.9pt);
		\fill[] (2.1,0.7) circle (0.9pt);
		\fill[] (0,1.4) circle (0.9pt);\fill[] (0.7,1.4) circle (0.9pt);\fill[] (1.4,1.4) circle (0.9pt);

		\node[right=0.1pt] at (2.1,1.7){$v$};
		\node at (0,-1/3){$u$};
		\node[rectangle] at (0.4,0.9) {$x_2$};
		\node[rectangle] at (1.1,0.9) {$x_2$};
		\node[rectangle] at (1.75,0.9) {$x_2$};
		
		\begin{scope}[xshift=1.0cm]
		
		\begin{scope}[xshift=2.2cm]

		\draw[step=0.7cm,color=lightgray] (0,0) grid(2.1,1.4);
		\draw [line width=0.7pt](0,0) -- (0,1.4)-- (0.7,1.4) -- (1.4,1.4) -- (2.1,1.4);
		
		\fill[black] (0,0) circle(1.5pt) ;
		\fill[black] (2.1,1.4) circle(1.5pt) ;
		
		\fill[] (0.7,0) circle (0.9pt);\fill[] (1.4,0) circle (0.9pt);\fill[] (2.1,0) circle (0.9pt);
		\fill[] (0,0.7) circle (0.9pt);\fill[] (0.7,0.7) circle (0.9pt);\fill[] (1.4,0.7) circle (0.9pt);
		\fill[] (2.1,0.7) circle (0.9pt);
		\fill[] (0,1.4) circle (0.9pt);\fill[] (0.7,1.4) circle (0.9pt);\fill[] (1.4,1.4) circle (0.9pt);

		\node[right=0.1pt] at (2.1,1.7){$v$};
		\node at (0,-1/3){$u$};
		
		\node[rectangle] at (0.4,1.6) {$x_3$};
		\node[rectangle] at (1.1,1.6) {$x_3$};
		\node[rectangle] at (1.75,1.6) {$x_3$};

		\begin{scope}[xshift=1.0cm]
		
		\begin{scope}[xshift=2.2cm]

		\draw[step=0.7cm,color=lightgray] (0,0) grid(2.1,1.4);
		\draw [line width=0.7pt](0,0) -- (0.7,0)-- (0.7,0.7) -- (1.4,0.7) -- (1.4,1.4)--(2.1,1.4);
		
		\fill[black] (0,0) circle(1.5pt) ;
		\fill[black] (2.1,1.4) circle(1.5pt) ;
		
		\fill[] (0.7,0) circle (0.9pt);\fill[] (1.4,0) circle (0.9pt);\fill[] (2.1,0) circle (0.9pt);
		\fill[] (0,0.7) circle (0.9pt);\fill[] (0.7,0.7) circle (0.9pt);\fill[] (1.4,0.7) circle (0.9pt);
		\fill[] (2.1,0.7) circle (0.9pt);
		\fill[] (0,1.4) circle (0.9pt);\fill[] (0.7,1.4) circle (0.9pt);\fill[] (1.4,1.4) circle (0.9pt);

		\node[right=0.1pt] at (2.1,1.7){$v$};
		\node at (0,-1/3){$u$};
		
		\node[rectangle] at (0.4,0.2) {$x_1$};
		\node[rectangle] at (1.1,0.9) {$x_2$};
		\node[rectangle] at (1.75,1.6) {$x_3$};
		
		\end{scope}
		\end{scope}
		\end{scope}
		\end{scope}
		\end{scope}
		\end{scope}
		\end{tikzpicture}
	\end{center}
	\caption{The four paths from $u$ to $v$ associated to $H_{3}^{(2)}(x_1,x_2,x_3)$.}
	\label{fg}
\end{figure}
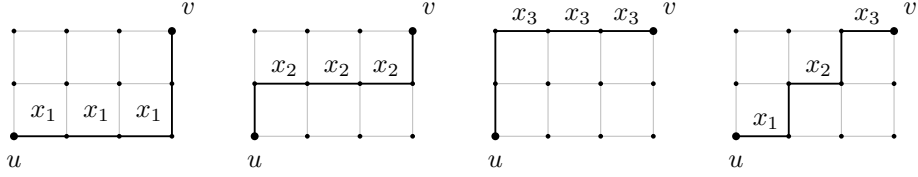

By setting $s = k$ in Theorem \ref{T4.4}, we will have the following result.

\begin{corollary}
Let $k,n$ be two positive integers and let $x_{1},x_{2},\ldots,x_{n}$ be independent variables. Then, the elementary symmetric function $e_k(x_{1},x_{2},\ldots,x_{n})$ is a weight-generating function of the paths between the points $u = (0, 0)$ and $v= (k, n-1)$ with
at most one vertex in the eastern direction.
\end{corollary}

As we can see in \cite{BAZ}, the generalized symmetric functions $E_k^{(s)}$ can be interpreted considering the set of all tilings of an $(n+k-1)$-board using exactly $k$ red squares and $n-1$ green squares with at most $s$ red squares successively. There is an obvious bijection between this tiling interpretation and the lattice path interpretation. This tiling interpretation for the generalized symmetric functions $E_k^{(s)}$ can be adapted to the generalized symmetric functions $H_k^{(s)}$ in the following way.

Let $\mathcal{T}^{s}_{n,k}$ be the set of all tilings of an $(n+k-1)$-board using exactly $k$ red squares and $n-1$ green squares such that the number of successive red squares is congruent to $0$ or $1$ modulo $s+1$. Also let $X^{w_{T}}=x_{1}^{w_{1}}x_{2}^{w_{2}}\cdots x_{n}^{w_{n}}$ be the weight of tiling $T$. For each $T\in \mathcal{T}^{s}_{n,k}$, we calculate $w_{T}=(w_{1},w_{2}\ldots,w_{n})$ as follows:
\begin{enumerate}
  \item Assign a weight to each individual square in the tiling. A green square always receives a weight of $1$. A red square has weight $x_{m+1}$ where $m$ is equal to the number of green squares to the left of that red square in the tiling.
  \item Calculate $w_{T}=(w_{1},w_{2}\ldots,w_{n})$ by multiplying the weight $x_{m+1}$ of all the red squares.
\end{enumerate}
We also consider

\medskip

$n_{m}\left(T\right):=$the number of successive red squares  modulo $(s+1)$ after the $m$-th green square to the left of these red squares in the tiling.

For example, the weight of the tiling rrgrg is $x_{1}^{1+1}x_{2}^{1}=x_{1}^{2}x_{2}$. Figure \ref{tqq} shows this tiling and its lattice path.

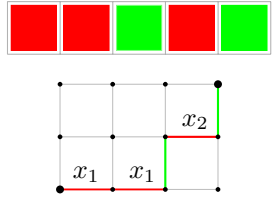
\begin{figure}[h]
	
	\begin{center}
		\begin{tikzpicture}
		
		\draw[step=0.7cm,color=black!30] (0,0) grid(3.5,0.7);
		\fill[fill=red,draw=red] (0.05,0.05) rectangle (0.65,0.65);
		\fill[fill=red,draw=red] (0.75,0.05) rectangle (1.35,0.65);
		\fill[fill=green,draw=green!50] (1.45,0.05) rectangle (2.05,0.65);
		\fill[fill=red,draw=red] (2.15,0.05) rectangle (2.75,0.65);
		\fill[fill=green,draw=green] (2.85,0.05) rectangle (3.45,0.65);
		
		\end{tikzpicture}\\
		\vspace{0.3cm}
		\begin{tikzpicture}

		\draw[step=0.7cm,color=lightgray] (0,0) grid(2.1,1.4);
		\node[rectangle] at (0.35,0.2) {$x_1$};
		\node[rectangle] at (1.1,0.2) {$x_1$};
		\node[rectangle] at (1.8,0.9) {$x_2$};
		\draw [red, line width=0.7pt](0,0) -- (1.4,0) ;
		\draw [green, line width=0.8pt] (1.4,0) -- (1.4,0.7);
		\draw [red, line width=0.8pt] (1.4,0.7)--(2.1,0.7);
		\draw [green, line width=0.8pt] (2.1,0.7)--(2.1,1.4);
		\fill[] (0.0,0) circle (1.5pt);
		\fill[] (0.7,0) circle (0.9pt);
		\fill[] (1.4,0) circle (0.9pt);
		\fill[] (2.1,0) circle (0.9pt);
		
		\fill[] (0.0,0.7) circle (0.9pt);
		\fill[] (0.7,0.7) circle (0.9pt);
		\fill[] (1.4,0.7) circle (0.9pt);
		\fill[] (2.1,0.7) circle (0.9pt);
		
		\fill[] (0.0,1.4) circle (0.9pt);
		\fill[] (0.7,1.4) circle (0.9pt);
		\fill[] (1.4,1.4) circle (0.9pt);
		\fill[] (2.1,1.4) circle (1.5pt);
		\end{tikzpicture}
		
	\end{center}
	\caption{A tiling of the weight $x_{1}^{2}x_{2}$ and its associated lattice path.}\label{tqq}
\end{figure}

\begin{theorem}\label{T4.5}
Let $k,n$ and $s$ be three positive integers and let $x_{1},x_{2},\ldots,x_{n}$ be independent variables. Then $H_{k}^{(s)}\left( x_{1},x_2,\ldots ,x_{n}\right)$ is created by summing the weights of all tilings of $\mathcal{T}^{s}_{n,k}$. That is,

\[
H_{k}^{(s)}\left( x_{1},x_2,\ldots ,x_{n}\right) =
\begin{cases}
\sum_{T\in \mathcal{T}^{s}_{n,k}} X^{w_{T}}, & \text{if }s\text{ odd,} \\
(-1)^{k}\sum_{T\in \mathcal{T}^{s}_{n,k}}(-1)^{G} X^{w_{T}}, &
\text{otherwise}
\end{cases}
\]
with $G=\sum_{T\in \mathcal{T}^{s}_{n,k}} n_{m}(T)$.
\end{theorem}

\begin{proof}
Since the bijection between lattice paths and tiling is weight-preserving. Then, from Theorem \ref{T4.4} it is suffice to associate a lattice path to each $(n+k-1)$-tiling using $k$ red squares and $n-1$ green squares with the number of successive red squares congruent to $0$ or $1$ modulo $(s+1)$. This lattice path starts from in $u=(0,0)$ and ends in $v=(k,n-1)$ where the number of the vertices in the eastern direction is congruent to $0$ or $1$ modulo $(s+1)$ whose each green tile represents a move one unit up and each red square represents a move one unit right.
\end{proof}

 Figure \ref{Fig5} shows the tiling interpretation for
$$H_{3}^{(2)}(x_1,x_2,x_3)=-x_{1}^{3}-x_{2}^{3}-x_{3}^{3}+x_1x_2x_3.$$
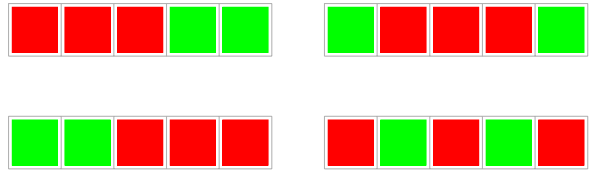
\begin{figure}[h]
	
	\begin{center}
		\begin{tikzpicture}
		
		\draw[step=0.7cm,color=black!30] (0,0) grid(3.5,0.7);
		\fill[fill=red,draw=red] (0.05,0.05) rectangle (0.65,0.65);
		\fill[fill=red,draw=red] (0.75,0.05) rectangle (1.35,0.65);
		\fill[fill=red,draw=red] (1.45,0.05) rectangle (2.05,0.65);
		\fill[fill=green,draw=green] (2.15,0.05) rectangle (2.75,0.65);
		\fill[fill=green,draw=green] (2.85,0.05) rectangle (3.45,0.65);
		
		\begin{scope}[xshift=1.5cm]
		\begin{scope}[xshift=2.7cm]
		
		\draw[step=0.7cm,color=black!30] (0,0) grid(3.5,0.7);
		\fill[fill=green,draw=green] (0.05,0.05) rectangle (0.65,0.65);
		\fill[fill=red,draw=red] (0.75,0.05) rectangle (1.35,0.65);
		\fill[fill=red,draw=red] (1.45,0.05) rectangle (2.05,0.65);
		\fill[fill=red,draw=red] (2.15,0.05) rectangle (2.75,0.65);
		\fill[fill=green,draw=green] (2.85,0.05) rectangle (3.45,0.65);
		
        \end{scope}
		\end{scope}

		\begin{scope}[yshift=-0.5cm]
		\begin{scope}[yshift=-1.0cm]
		
		\draw[step=0.7cm,color=black!30] (0,0) grid(3.5,0.7);
		\fill[fill=green,draw=green] (0.05,0.05) rectangle (0.65,0.65);
		\fill[fill=green,draw=green] (0.75,0.05) rectangle (1.35,0.65);
		\fill[fill=red,draw=red] (1.45,0.05) rectangle (2.05,0.65);
		\fill[fill=red,draw=red] (2.15,0.05) rectangle (2.75,0.65);
		\fill[fill=red,draw=red] (2.85,0.05) rectangle (3.45,0.65);
		
		   		\begin{scope}[xshift=1.5cm]
		   \begin{scope}[xshift=2.7cm]
		
		   \draw[step=0.7cm,color=black!30] (0,0) grid(3.5,0.7);
		   \fill[fill=red,draw=red] (0.05,0.05) rectangle (0.65,0.65);
		   \fill[fill=green,draw=green] (0.75,0.05) rectangle (1.35,0.65);
		   \fill[fill=red,draw=red] (1.45,0.05) rectangle (2.05,0.65);
		   \fill[fill=green,draw=green] (2.15,0.05) rectangle (2.75,0.65);
		   \fill[fill=red,draw=red] (2.85,0.05) rectangle (3.45,0.65);
		
		   \end{scope}
		   \end{scope}
		
        \end{scope}
        \end{scope}
  		\end{tikzpicture}
	\end{center}
\caption{The four tilings associated to $H_3^{(2)}(x_1,x_2,x_3)$.}\label{Fig5}
\end{figure}

By Theorem \ref{T4.5}, we can also interpret the elementary symmetric function as follows.

\begin{corollary}
Let $k,n$ be two positive integers and let $x_{1},x_{2},\ldots,x_{n}$ be independent variables. Then, the elementary symmetric function $e_k(x_{1},x_{2},\ldots,x_{n})$ is a weight-generating function of all tilings of an $(n+k-1)$-board using exactly $k$ red squares and $n-1$ green squares with at most one red square successively.
\end{corollary}

\section{Concluding remarks}

In this paper, we investigate a pair of two symmetric functions which generalize the complete and elementary symmetric functions.
We show that these generalized symmetric functions satisfy many of the classical relations between complete and elementary symmetric functions.
Most of these relationships have the same shape.

The Schur symmetric functions $s_\lambda(x_1,x_2,\ldots,x_n)$ for a partition $\lambda$ can be extended in the same way. For example, we can define the generalized Schur symmetric function $s_{\lambda}^{(s)}=s_{\lambda}^{(s)}(x_1,x_2,\ldots,x_n)$ in terms of the generalized symmetric functions $H_k^{(s)}=H_k^{(x)}(x_1,x_2,\ldots,x_n)$ or $E_k^{(s)}=E_k^{(s)}(x_1,x_2,\ldots,x_n)$ as follows:
$$ s_{\lambda}^{(s)} := \det(H_{\lambda_i-i+j}^{(s)})_{1\le i,j \le n} $$
or
$$ s_{\lambda}^{(s)} := \det(E_{\lambda'_i-i+j}^{(s)})_{1\le i,j \le n},$$
where $\lambda'$ is the conjugate of $\lambda$.

It would be very appealing to
investigate the properties of the generalized Schur symmetric functions $s_{\lambda}^{(s)}$.

\bigskip


\bibliographystyle{amsplain}

\begin{thebibliography}{10}
\bibitem{BA0} Andrews, G.E.: \textit{The theory of partitions}. Addison-Wesley Publishing, New York (1976)

\bibitem{BAZ} Bazeniar, A., Ahmia, M., Belbachir, H.: \textit{Connection between bi$^{s}$nomial coefficients with their analogs and symmetric functions}. Turk J Math. \textbf{42}, 807--818 (2018)

\bibitem{bsb} Belbachir, H., Bouroubi, S., Khelladi, A.: \textit{Connection between ordinary multinomials, Fibonacci numbers, Bell polynomials and discrete uniform distribution}. Ann. Math. Inform. \textbf{35}, 21--30 (2008)


\bibitem{Fu} Fu. H., Mei. Z.: \textit{Truncated homogeneous symmetric functions},
Linear Multilinear Algebra. (2020) DOI: 10.1080/03081087.2020.1733460

\bibitem{Gri} Grinberg, D.: Petrie symmetric functions, arXiv:2004.11194v1.

\bibitem{Gou} Gould, H.W.:  \textit{The Girard-Waring power sum formulas for symmetric functions and Fibonacci sequences}. Fibonacci Quart. \textbf{37}(2), 135--140 (1999)

%
\bibitem{Mck} Macdonald, I.: \textit{Symmetric functions and Hall polynomials}. Oxford Univ Press, Oxford (1979)

\bibitem{Mac} MacMahon, P.A.: \textit{Combinatory analysis, vol. 2}. Chelsea Publishing Co., New York (1960). (bound as one)

\bibitem{Mer1}  Merca, M.: \textit{Fast algorithm for generating ascending compositions}. J. Math. Model. Algorithms \textbf{11}, 89--104 (2012)

\bibitem{Mer2} Merca, M.: \textit{Binary diagrams for storing ascending compositions}. Comput. J. \textbf{56}(11), 1320--1327 (2013)

\bibitem{Mer14} Merca, M.: \textit{A generalization of the symmetry between complete and elementary symmetric functions}. Indian J. Pure Appl. Math., \textbf{45}(1), 75--89 (2014)

\bibitem{Mer16a} Merca, M: \textit{New convolutions for complete and elementary symmetric functions}. Integral Transforms Spec. Funct. \textbf{27}(12), 965--973 (2016)
%
%

\bibitem{Mer5} Merca, M.: \textit{Bernoulli numbers and symmetric functions}. Rev. R. Acad. Cienc. Exactas F\'{i}s. Nat. Ser. A Math. RACSAM. \textbf{114}(1), 20 (2020)

%
\end{thebibliography}

\end{document}